\documentclass[12pt]{article}
\usepackage{graphicx}
\usepackage{tikz}

\usepackage{palatino}
\usepackage[total={7in, 9.2in}]{geometry}
\usepackage{amsmath, amssymb, amsthm}
\usepackage[square,sort,comma,numbers]{natbib}
\usepackage{enumitem}
\usepackage{array}
\usepackage{booktabs}

\usepackage[titles]{tocloft}
\theoremstyle{plain}
\newtheorem{theorem}{Theorem}[section]
\newtheorem*{theoremIntro}{Theorem}
\newtheorem{proposition}[theorem]{Proposition}
\newtheorem{lemma}[theorem]{Lemma}
\newtheorem{corollary}[theorem]{Corollary}
\newtheorem{conjecture}[theorem]{Conjecture}

\theoremstyle{definition}
\newtheorem{definition}[theorem]{Definition}
\newtheorem{example}[theorem]{Example}

\theoremstyle{remark}
\newtheorem{remark}[theorem]{Remark}
\newtheorem{question}[theorem]{Question}

\newcommand{\CC}{\mathbb C}
\newcommand{\PP}{\mathbb P}
\newcommand{\QQ}{\mathbb Q}
\newcommand{\RR}{\mathbb R}
\newcommand{\TT}{\mathbb T}
\newcommand{\ZZ}{\mathbb Z}

\newcommand{\MR}{\mathcal R}

\DeclareMathOperator{\GL}{GL}

\DeclareMathOperator{\EES}{EES}
\DeclareMathOperator{\sym}{Sym}
\DeclareMathOperator{\diag}{Diag}
\DeclareMathOperator{\Hom}{Hom}
\DeclareMathOperator{\Proj}{Proj}
\DeclareMathOperator{\rank}{rank}
\DeclareMathOperator{\Ehr}{Ehr}
\DeclareMathOperator{\hgt}{ht}
\DeclareMathOperator{\Span}{span}
\DeclareMathOperator{\aff}{aff}

\DeclareMathOperator{\conv}{conv}
\DeclareMathOperator{\Fix}{Fix}
\DeclareMathOperator{\Vol}{Vol}

\DeclareMathOperator{\tr}{tr}
\newcommand{\Ev}{N_{\mathrm{even}}}
\newcommand{\Od}{N_{\mathrm{odd}}}
\newcommand{\orb}[1]{\TT_{#1}}

\usepackage[bookmarksnumbered=true]{hyperref}
\hypersetup{
     colorlinks = true,
     linkcolor = blue,
     anchorcolor = blue,
     citecolor = red,
     filecolor = blue,
     urlcolor = blue
     }

\title{The equivariant Ehrhart theory of polytopes\\ with order-two symmetries: Part II}
\author{Oliver Clarke}
\date{October 2026}

\begin{document}

\maketitle

\begin{abstract}
We give two counterexamples to Stapledon's Effectiveness Conjecture. Each counterexample is a $6$-dimensional lattice polytope with an order-two symmetry whose equivariant $H^*$-series $\varphi[t]$ is a polynomial with one coefficient that is not effective.
For polytopes with a symmetry of order two, the goal of this paper is to build up a picture of the landscape of results for when the Effectiveness Conjecture holds.
Our main result is a criterion, in terms of the faces of $P$, that determines when the Ehrhart ring of $P$ admits an invariant linear system of parameters of the correct representation type.
We use these results to reduce the space of possible counterexamples to one of two cases. A randomised computer search produced counterexamples in each case.
We also give an example where the equivariant $H^*$-series is effective but admits no invariant linear system of parameters, neither for $P$ nor any invariant subdivision of it, answering a question of Stapledon in the case of groups of order two. 
\end{abstract}

\tableofcontents

\section{Introduction}\label{sec:intro}

Equivariant Ehrhart theory, introduced by Stapledon \cite{Sta11}, counts the lattice points of dilates of a symmetric lattice polytope $P$, and records them as permutation representations of a finite group $G$ acting on $P$. The resulting generating function is the \emph{equivariant Ehrhart series} $\EES_P^G(t)$. We consider $\EES_P^G(t)$ as the product of the \emph{equivariant $H^*$-series} $\varphi[t]$ and some generating function generalising the Ehrhart series of the standard simplex. The equivariant $H^*$-series is a power series whose coefficients are virtual representations of $G$, and is the equivariant generalisation of the classical $h^*$-polynomial $h^*_P$. However, unlike $h^*_P$, the series $\varphi[t]$ need not be a polynomial, and its coefficients need not be \emph{nonnegative}, which, in the language of representations, corresponds to a virtual character being \emph{effective}. In this paper, we will give counterexamples to the conjecture that $\varphi[t]$ is polynomial if and only if it is effective, usually called the \emph{Effectiveness Conjecture}.

\medskip

\noindent \textbf{Brief setup.}
The geometric setup for equivariant Ehrhart theory is displayed in Figure~\ref{fig:setup} and explicitly given in Section~\ref{setup}. The polytope $P$ lies in an affine hyperplane at height one in $\RR^{n+1}$ so that the lattice points of the $k$-th dilate of $P$ are exactly the lattice points of the cone over $P$ at height $k$. A group $G \subseteq \GL_{n+1}(\ZZ)$ that leaves $P$ invariant acts linearly on $\RR^{n+1}$ and induces an action on each affine subspace $\widetilde M_k$ at height $k$. In this setup, the denominator of the equivariant Ehrhart series $\det(I - t\rho)$ is the equivariant version of dividing the Ehrhart ring of $P$ by a \emph{linear system of parameters}. Theorem~\ref{thm:lsop} explains when such a division is possible when $G$ is group of order two. The Effectiveness Conjecture asserts that if $\varphi[t]$ is a polynomial, then every coefficient of $\varphi[t]$ is the character of a genuine representation of $G$.

\begin{figure}[t]
\centering
\begin{tikzpicture}[scale=1.25]
  \fill[gray!12] (0,0) -- (-2.0,3.0) -- (4.0,3.0) -- cycle;
  \draw[gray!60] (0,0) -- (-2.0,3.0);
  \draw[gray!60] (0,0) -- (4.0,3.0);
  \foreach \y in {0,1,2} {\foreach \x in {-3,-2,-1,0,1,2,3,4,5} {\fill[gray!55] (\x/1.5,\y) circle (0.04);}}
  \draw[dashed,gray] (-2.3,0) -- (4.0,0) node[right,black]{\small $\widetilde M_0$};
  \draw[dashed,gray] (-2.3,1) -- (4.0,1) node[right,black]{\small $\widetilde M_1$};
  \draw[dashed,gray] (-2.3,2) -- (4.0,2) node[right,black]{\small $\widetilde M_2$};
  \draw[very thick] (-2/3,1) -- (4/3,1);
  \foreach \x in {-1,0,1,2} {\fill (\x/1.5,1) circle (0.05);}
  \node at (-0.95,1.2) {\small $P$};
  \draw[very thick,gray!70] (-4/3,2) -- (8/3,2);
  \foreach \x in {-2,-1,0,1,2,3,4} {\fill[gray!70] (\x/1.5,2) circle (0.05);}
  \node at (-1.62,2.2) {\small $2P$};
  \draw[thick,dotted] (0,0) -- (1,3);
  \fill (0,0) circle (0.05) node[below]{\small $0$};
  \draw[->,thick] (0,0) -- (2/3,2);
  \node at (0.34,1.78) {\small $\alpha$};
\end{tikzpicture}
\caption{The setup: the polytope $P$ sits at height one, its dilate $kP$ at height $k$, and the group acts linearly preserving heights. The dotted line spanned by $\alpha$ is a distinguished one-dimensional fixed space of the action.}
\label{fig:setup}
\end{figure}
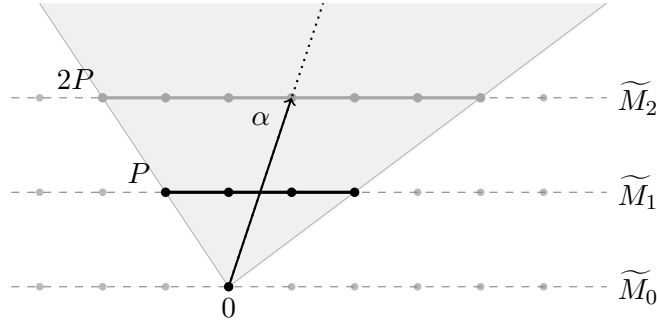

This note follows \cite{CHK} and explores the case where $G$ has order two. Writing the generator $\sigma$ of $G$ in Reiner's normal form \cite{Rei57}, see Section~\ref{sec:order-two}, we define $p$ and $q$ to be the dimensions of the $(+1)$ and $(-1)$-eigenspaces of $\sigma$, respectively. Our results below are expressed in terms of $p$, $q$ and the fixed polytope $P^\sigma = \{x \in P \colon \sigma(x) = x\}$.

\medskip

\noindent \textbf{Counterexamples.}
Proposition~\ref{prop:open} characterises possible counterexamples to the Effectiveness Conjecture into two types (C1) and (C2). We exhibit counterexamples of each type: Example~\ref{ex:counterexample} and \ref{ex:counterexample-c2}, respectively. Both are 
$6$-dimensional lattice polytopes with eight vertices.

\begin{theoremIntro}[]
There is a $\sigma$-invariant lattice polytope whose equivariant $H^*$-series is a polynomial but is not effective. Example~\ref{ex:counterexample} is a $6$-dimensional polytope $P$ with vertices
\[
\begin{array}{llll}
(1,0,0,0,0,0,1), & (0,0,1,0,0,0,1), & (0,1,0,1,1,0,1), & (0,0,1,1,1,0,1),\\
(0,0,0,1,0,0,1), & (0,0,0,0,0,1,1), & (1,1,0,0,1,0,1), & (1,1,0,0,0,1,1),
\end{array}
\]
where $\sigma(u,v,h) = (v,u,h)$ for $u,v \in \RR^3$ and $h \in \RR$. So in this case $p = 4$ and $q = 3$. Let $\chi_0$ and $\chi_1$ be the trivial and sign characters of $G = \langle \sigma \rangle$. Then the equivariant $H^*$-series of $P$ with respect to $G$ is 
\[
\varphi[t] = (1+t^2+t^4)\chi_0 + (t-t^2+t^3)\chi_1,
\]
where the coefficient $\chi_0 - \chi_1$ of $t^2$ is not effective. Example~\ref{ex:counterexample-c2} is another $6$-dimensional polytope where the action of $\sigma$ has $p = 5$ and
$q = 2$, and $\varphi[t] = (1+t^2+t^3)\chi_0 + (t-t^2+t^3)\chi_1$.
\end{theoremIntro}

We found these examples by reducing the unknown cases of Conjecture~\ref{conj:eff} to two cases followed by a randomised search.

\medskip

\noindent \textbf{Overview and results.}
In Section~\ref{sec:order-two}, we fix our notation and give an exact criterion for $\varphi[t]$ to be a polynomial in terms of the two constituents of the Ehrhart quasi-polynomial of $P^\sigma$. See Proposition~\ref{prop:poly}.
In Section~\ref{sec:known}, we recall Stapledon's hierarchy of conditions and explain the cases of the Effectiveness Conjecture that are known in the literature.
In Section~\ref{sec:main}, we prove our main result, Theorem~\ref{thm:lsop}, which characterises the existence of a certain invariant linear system of parameters by a condition on the faces of $P$. We say that a face $F$ is \emph{$\sigma$-odd} if it is $\sigma$-invariant and every $\sigma$-fixed element of the group generated by the lattice points of $F$ has even height, see Definition~\ref{def:odd}. We call a $\sigma$-odd face $F$ \emph{obstructing} if $\dim F^\sigma \ge q$.

\begin{theoremIntro}[Theorem~\ref{thm:lsop}]
Let $\rho : G \rightarrow \GL_{n+1}(\ZZ)$ be the natural representation of $G$. The Ehrhart ring $S_P$ has a $\sigma$-stable linear system of parameters $V \subseteq (S_P)_1$ of class $[V] = [\rho\otimes\CC]$ if and only if $P$ has no obstructing face. In that case $\varphi[t]$ is the graded character of $S_P/(V)$ and, in particular, it is effective.
\end{theoremIntro}

From this we deduce the Effectiveness Conjecture in three cases, see Corollaries~\ref{cor:p-le-q}, \ref{cor:lattice-fixed} and \ref{cor:p-eq-q-plus-1}, and we recover a combinatorial criterion of Stapledon (for a group of order two) whose published proof is incomplete, see Remark~\ref{rem:cor710}. Table~\ref{tab:status} shows the status of the conjecture.
In Section~\ref{sec:examples}, we provide examples to highlight the definitions and results. In particular, Example~\ref{ex:TxI} shows that effectiveness does not imply the existence of a linear system of parameters with the correct representation type, and Examples~\ref{ex:counterexample} and \ref{ex:counterexample-c2} show that polynomiality does not imply effectiveness. 
In Section~\ref{sec:open}, we isolate the two cases where the conjecture can fail and give two reductions towards them using monotonicity and dilations. Firstly, in Corollary~\ref{cor:few-vertices}, we reduce the conjecture to polytopes with few vertices. Secondly, in Proposition~\ref{prop:dilate}, we give a reduction to odd dilates. These results reduce the search space for a counterexample to a smaller set which allows us to carry out the search in Section~\ref{sec:search}. The counterexamples lie in each branch of Proposition~\ref{prop:open}.
In Section~\ref{sec:proofs} and give two small examples illustrating the argument.
In Section~\ref{sec:conjectures}, we consider the $h^*$-polynomial $H$ of the fixed polytope $P^\sigma$ and give a inequality for its coefficients. We then refine Theorem~\ref{thm:lsop} to an obstruction at a single face, which allows us to apply it to an invariant subdivision.

\begin{theoremIntro}[Corollary~\ref{cor:q433} and Example~\ref{ex:TxI}]
There is a $\sigma$-invariant lattice polytope $P$ whose equivariant $H^*$-series is effective but neither its Ehrhart ring nor the toric face ring of any $\sigma$-invariant lattice polyhedral subdivision of $P$ has a $\sigma$-stable linear system of parameters of class $[\rho\otimes\CC]$.
\end{theoremIntro}

This gives an answer to Stapledon's Question~\ref{q:433} for groups of order two. We end Section~\ref{sec:conjectures} with two conjectures and a question. The first conjecture, Conjecture~\ref{conj:ineq}, is a coefficient-wise inequality between $h^*_P$ and the numerator $H$, which is weaker than effectiveness but remains open after the examples presented in this paper.

\section{Setup}\label{sec:complete_setup}

We begin by giving the setup for equivariant Ehrhart theory before specialising to the case of groups of order two, which is the setup we will assume throughout the rest of the paper.

\subsection{Setup for equivariant Ehrhart theory}
\label{setup}
Throughout fix a vector $\alpha \in \ZZ^{n+1}$ and a finite group $G \subseteq \GL_{n+1}(\ZZ)$ such that $g(\alpha) = \alpha$ for all $g \in G$. Let $\widetilde M_0 \subseteq \RR^{n+1}$ be a $G$-invariant $n$-dimensional subspace of $\RR^{n+1}$ not containing $\alpha$ such that $M_0 := \widetilde M_0 \cap \ZZ^{n+1} \cong \ZZ^n$ is a rank $n$ sublattice. Let $\mathbf a \in \ZZ^{n+1}$ be a representative of a generator of the quotient $\ZZ^{n+1} / M_0 \cong \ZZ$ and for each $k \in \ZZ$ define the affine lattice at height $k$ to be $M_k = k \mathbf a + M_0$, with corresponding affine space $\widetilde M_k = k\mathbf a + \widetilde M_0$. For $k = 1$ we write $M = M_1$ and $\widetilde M = \widetilde M_1$. The \emph{height} of a lattice vector $v \in \ZZ^{n+1}$ is the integer $\hgt(v) = k$ such that $v$ lies in the affine lattice $M_k$.

We write $\MR(G)$ for the representation ring of $G$, whose elements are formal $\ZZ$-linear combination of representations, where addition is given by direct sums and multiplication by tensor products. By an abuse of notation, we identify a representation $\rho : G \rightarrow \GL_{n+1}(\CC)$ with its class $[\rho] \in \MR(G)$ and also with its character $\chi := (g \mapsto \tr \rho(g))$. An element of $\MR(G)$ is \emph{effective} if it is the class of a genuine representation, that is, a nonnegative combination of irreducible characters. Since our group $G$ typically lies in $\GL_{n+1}(\ZZ)$, we will often write $\rho : G \rightarrow \GL_{n+1}(\ZZ)$ for the inclusion map and write $\rho \otimes \CC$ for the map $\rho$ where the entries of $\rho (g)$ are considered as elements of $\CC$.

Let $P \subseteq \widetilde M$ be an $n$-dimensional $G$-invariant lattice polytope, so the vertices of $P$ lie in the affine lattice $M$ and every group element $g$ satisfies $g(P) = P$. For $k \ge 0$, let $\chi(k) \in \MR(G)$ be the class of the permutation representation of $G$ acting on the lattice points $kP \cap M_k$, where $0P \cap M_0 = \{0\}$. The \emph{equivariant Ehrhart series} is
\[
\EES_P^G(t) = \sum_{k \ge 0}  \chi(k) t^k = \frac{\varphi[t]}{\det(I - t \rho)} = \varphi[t] \sum_{k \ge 0} \sym^k (\rho)\, t^k ,
\]
where $\rho\colon G \to \GL_{n+1}(\ZZ)$ is the natural representation as above and the denominator is the formal alternating sum of alternating representations $\det(I - t\rho) = \sum_i (-1)^i [\bigwedge^i \rho]\,t^i$. See \cite[Lemma~3.1]{Sta11} for the right-most equality. The power series $\varphi[t] \in \MR(G)[[t]]$ is the \emph{equivariant $H^*$-series}. We call $\varphi[t]$ \emph{polynomial} if it lies in $\MR(G)[t]$ and we say $\varphi[t]$ is \emph{effective} if each of its coefficients is effective.

Since the vector $\alpha$ is fixed by $G$ and lies outside the subspace $\widetilde M_0$, its image generates a finite-index subgroup of the quotient $\ZZ^{n+1}/M_0 \cong \ZZ$ and is fixed by the group. Since the only automorphisms of the integers are $\pm1$, the group acts trivially on that quotient. Hence the height map $\hgt$ is a $G$-invariant homomorphism, every affine lattice $M_k$ is $G$-invariant, and the lattice points of each dilate satisfy $kP \cap \ZZ^{n+1} = kP \cap M_k$.

\medskip\noindent\textbf{The Ehrhart ring.} Let $C_P = \RR_{\ge0}P$ be the cone over $P$ and consider the affine semigroup ring
\[
S_P := \CC[C_P \cap \ZZ^{n+1}] \quad \text{with grading} \quad \deg x^u = \hgt(u).
\]
Then the degree-$k$ part $(S_P)_k$ has basis $\{x^u \colon u \in kP \cap M_k\}$, with an action by $G$ via $g\cdot x^u = x^{g(u)}$. So the series $\EES_P^G(t)$ is the equivariant Hilbert series of the ring $S_P$. By Hochster's theorem \cite{Hoc72}, $S_P$ is a Cohen--Macaulay normal domain of Krull dimension $n+1$. A \emph{linear system of parameters} (l.s.o.p.) is a subspace $V \subseteq (S_P)_1$ of dimension $n+1$ such that $S_P/(V)$ is a finite-dimensional $\CC$-vector space. 
The action of $G$ preserves the grading, so $G$ acts on the degree-one part $(S_P)_1$. We say that a subspace $V \subseteq (S_P)_1$ is \emph{$G$-stable} if $g(V) = V$ for all $g \in G$, and write $[V] \in \MR(G)$ for the class of the resulting representation. In the case $G = \langle\sigma\rangle$, we write \emph{$\sigma$-stable} instead of $G$-stable.
Recall that the \emph{Ehrhart series} of $P$ is the generating function $\Ehr_P(t) = \sum_{k\ge0}\#(kP\cap M_k)\,t^k$, and that the \emph{$h^*$-polynomial} $h^*_P$ of $P$ is the numerator of the Ehrhart series
\[
\Ehr_P(t) \ =\ \frac{h^*_P(t)}{(1-t)^{n+1}}.
\]
By a well-known result of Stanley, $h^*_P$ is a polynomial of degree at most $n$ with nonnegative integer coefficients \cite{Sta80}. For a polynomial $f$ we write $[f(t)]_i$ for the coefficient of $t^i$ in $f$, and we also write $(h^*_P)_i = [h^*_P(t)]_i$. Since the Ehrhart ring is Cohen--Macaulay, a basis of $V$ forms a regular sequence \cite[Theorem~2.1.2]{BH}, so the Hilbert series of the quotient is $h^*_P$ \cite{Sta78}. Proposition~\ref{prop:lsop-eff} is the equivariant analogue.

Note that elements of the representation ring $\MR(G)$ are determined by their characters. The value $\chi(k)(g)$ is the number of lattice points of $kP$ fixed by $g$ \cite[Lemma~5.2]{Sta11}. If $g = e$ is the identity, then we recover the usual $h^*$-polynomial $\varphi[t](e) = h^*_P(t)$. Observe that if $\varphi[t]$ is effective then it is a polynomial \cite[Remark~2.13]{Sta23} because the dimension of the $i$-th coefficient (i.e., the coefficient's evaluation at $e$) of $\varphi[t]$ is the $i$-th coefficient of the polynomial $h^*_P$, which vanishes for $i > n$ and the only effective representation of dimension zero is zero.

\begin{conjecture}[Effectiveness Conjecture {\cite[Conjecture~1.2]{Sta23}}]\label{conj:eff}
If $\varphi[t]$ is polynomial, then $\varphi[t]$ is effective.
\end{conjecture}

Stapledon's original conjecture \cite[Conjecture~12.1]{Sta11} asked for more: if $\varphi[t]$ is polynomial then there exists a $G$-invariant nondegenerate polynomial. This stronger version is false, by the counterexample of Santos and Stapledon discussed in Section~\ref{sec:known}. The \emph{Effectiveness Conjecture}, as it is currently known, is Conjecture~\ref{conj:eff}. Examples~\ref{ex:counterexample} and \ref{ex:counterexample-c2} shows that it does not hold.

We note that Section~\ref{setup} is Stapledon's setup in homogeneous form. By writing the lattice as $\ZZ^{n+1} = M_0 \oplus \ZZ\mathbf a$, the group acts on $M$ by affine transformations. The lattice $\ZZ^{n+1}$ plays the role of Stapledon's lattice $M \oplus \ZZ$ and the equivariant $H^*$-series $\varphi[t]$ is the same. Conversely any affine action arises this way by taking the vector $\alpha = \sum_{g \in G}g(u)$ for a lattice point $u \in M$. A difference in notation is the extra $(1-t)$ in the denominator of the equivariant Ehrhart series \cite[Section~1]{Sta11}, but that factor is exactly the copy of the trivial representation $\langle \alpha \rangle_\CC$ that is \emph{factored out} in the affine case.

\subsection{Symmetries of order two}\label{sec:order-two}

The main result of this section is Proposition~\ref{prop:poly}, which gives an exact criterion for $\varphi[t]$ to be a polynomial. First, we fix our notation.

From now on the group $G = \langle \sigma \rangle$ has order two. By Reiner's classification of integral representations \cite{Rei57}, there is a matrix $B \in \GL_{n+1}(\ZZ)$ and nonnegative integers $a,b,c$ such that
\[
B \sigma B^{-1} =
\diag(\underbrace{1, \dots, 1}_{a}, \underbrace{-1, \dots, -1}_{b}, \underbrace{U, \dots, U}_c),
\qquad
U = \begin{pmatrix}
    0 & 1 \\
    1 & 0
\end{pmatrix}.
\]
We assume throughout that the involution $\sigma$ has this shape, replacing the vector $\alpha$, the subspace $\widetilde M_0$ and the polytope $P$ by their images under the matrix $B$. Note that this changes none of the invariants below. We call the triple $(a,b,c)$ the \emph{type} of $\sigma$ and define
\[
p := a+c, \qquad q := b+c,
\]
which are the dimensions of the $(+1)$ and $(-1)$-eigenspaces of $\sigma$. In particular $p$ and $q$ do not depend on $B$ and satisfy $p + q = n+1$, with $p \ge 1$ because $\alpha$ is fixed and $q \ge 1$ because $\sigma$ is nontrivial. As a representation, the natural representation decomposes as $\rho \otimes \CC \cong p\chi_0 + q\chi_1$, where the characters $\chi_0$ and $\chi_1$ are the trivial and the sign characters. The representation ring is $\MR(G) = \ZZ[\chi_1]/\langle\chi_1^2-1\rangle$ where $\chi_0 \cong 1 \in \MR(G)$ and
\[
\det(I - t\sigma) = (1-t)^{p}(1+t)^{q}, \qquad \det(I - te) = (1-t)^{n+1}.
\]

\medskip

\noindent \textbf{The fixed polytope.}
Let $\Fix(\sigma) \subseteq \RR^{n+1}$ be the fixed space and $\pi = \frac12(I + \sigma) : \RR^{n+1} \rightarrow \Fix(\sigma)$ be the projection onto $\Fix(\sigma)$. Define
\[
\Lambda := (\ZZ^{n+1})^\sigma = \ZZ^{n+1} \cap \Fix(\sigma)
\]
to be the fixed lattice, which has rank $p$. We call $\hgt(\Lambda) \subseteq \ZZ$ the \emph{height group} of the action. For any subset $S$ of $\RR^{n+1}$ we write $S^\sigma = S\cap\Fix(\sigma)$ for its $\sigma$-fixed part. For instance $\Lambda = (\ZZ^{n+1})^\sigma$. We fix some coordinate labels: $x_1, \dots, x_a$ corresponding to the $(+1)$-block of $\sigma$; $y_1, \dots, y_b$ for the $(-1)$-block; and $z_1, \dots, z_c, w_1, \dots, w_c$ for the $U$-block such that $\sigma$ swaps $z_j$ and $w_j$ for each $j$. The fixed space $\Fix(\sigma)$ is cut out by the equations $y_i = 0$ and $z_j = w_j$ and we have
\[
P^\sigma = P \cap \Fix(\sigma), \qquad \pi(x,y,z,w) = \big(x,\,0,\,\frac{z+w}{2},\,\frac{z+w}{2}\big).
\]
We write $N(k) := \#(kP^\sigma \cap \ZZ^{n+1})$ for the number of $\sigma$-fixed lattice points of $kP$, and recall that its Ehrhart series is $\Ehr_{P^\sigma}(t) := \sum_{k\ge0}N(k)t^k$. The two evaluations of the equivariant $H^*$-series are
\begin{equation}\label{eq:evals}
\varphi[t](e) = h^*_P(t)
\quad \text{and}\quad 
\varphi[t](\sigma) = (1-t)^{p}(1+t)^{q}\,\Ehr_{P^\sigma}(t).
\end{equation}

\begin{lemma}[Fixed polytope]\label{lem:fixed}
Let $\pi = \frac12(I+\sigma)$ as above.
\begin{enumerate}[label=(\alph*)]
    \item The fixed polytope satisfies $P^\sigma = \pi(P)$, and its vertices are of the form $\pi(v)$ for some vertex $v$ of $P$. In particular $2P^\sigma$ is a lattice polytope and $P^\sigma$ is a lattice polytope if $c = 0$.
    \item If a face $F$ of $P$ is $\sigma$-invariant, then $F^\sigma = \pi(F)$ is a nonempty face of $P^\sigma$ meeting the relative interior of $F$, and its dimension is $\dim F^\sigma = \dim(\Span_\RR F)^\sigma - 1$.
    \item The fixed polytope has dimension $\dim P^\sigma = p-1$, and every proper $\sigma$-invariant face $F$ of $P$ satisfies $\dim F^\sigma \le p-2$.
    \item The height group defined by $\hgt(\Lambda)$ is either $\ZZ$ or $2\ZZ$.
\end{enumerate}
\end{lemma}

\begin{proof}
(a) Since the group has order two, by \cite[Lemma~5.4]{Sta11}, we have that $P^\sigma = \pi(P)$ and its vertices are of the form $\pi(v)$ for some vertex $v$ of $P$. Each vertex $\pi(v)$ of the fixed polytope is the midpoint of the pair $v$ and $\sigma(v)$, so every coordinate of $\pi(v)$ lies in $\frac12\ZZ$. Thus $2P^\sigma$ is a lattice polytope. The final claim that $P^\sigma$ is a lattice polytope if $c=0$ follows directly from the formula for $\pi$.

(b) Let $F$ be a $\sigma$-invariant face of $P$. By viewing $F$ as a $\sigma$-invariant lattice polytope in its affine span $\aff(F)$, by part (a), we have $F^\sigma = \pi(F)$. Write
$F = P \cap \{\ell = \max_P\ell\}$ for a supporting linear functional $\ell : \RR^{n+1} \rightarrow \RR$. Since the face $F$ is $\sigma$-invariant, the involution $\sigma$ permutes the vertices of $F$, so the average of those vertices is a $\sigma$-fixed point lying in the relative interior of $F$. Hence $F^\sigma$ is nonempty and meets the relative interior of $F$. Each point $z$ of $F^\sigma$ lies in $P^\sigma$ and satisfies $\ell(z) = \max_P\ell$, so the maximum of $\ell$ over the smaller polytope $P^\sigma$ is again $\max_P\ell$. Thus $F^\sigma = P^\sigma \cap \{\ell = \max_P\ell\}$, and so $F^\sigma$ is
a face of $P^\sigma$. Finally, by \cite[Lemma~5.3]{Sta11}, it follows that 
$\dim F^\sigma = \dim(\aff(F)\cap\Fix(\sigma))$.
Since $F$ lies at height one,
its affine span is $\aff(F) = \Span_\RR F \cap \{\hgt = 1\}$. So
$\aff(F)\cap\Fix(\sigma) = (\Span_\RR F)^\sigma \cap \{\hgt = 1\}$ is a
nonempty affine hyperplane of $(\Span_\RR F)^\sigma$ and so
$\dim F^\sigma = \dim(\Span_\RR F)^\sigma - 1$.

(c) By assumption $P$ is full dimensional, so by taking $F = P$ in part (b), we obtain $\dim P^\sigma = \dim(\RR^{n+1})^\sigma - 1 = p-1$. Let $F$ be a $\sigma$-invariant proper face of $P$. If $\dim F^\sigma = p-1$, then $F^\sigma$ is a full-dimensional face of $P^\sigma$, hence $F^\sigma = P^\sigma$. By part (b) applied to the polytope $P$, the fixed polytope $P^\sigma$ meets the relative interior of $P$. Since $F^\sigma = P^\sigma$ and $F^\sigma \subseteq F$, the face $F$ meets the relative interior of $P$, which contradicts the assumption that $F$ is a proper face of $P$. So $\dim(F^\sigma) \le p-2$.

(d) For any $u \in M$, we have that $u + \sigma u$ lies in the fixed lattice $\Lambda$ and has height $2$, so $2\ZZ \subseteq \hgt(\Lambda) \subseteq \ZZ$, i.e., $\hgt(\Lambda)$ is either $\ZZ$ or $2\ZZ$.
\end{proof}

Both cases in part (d) occur in small examples. Let $P = [0,1]\times\{1\} \subseteq \RR^2$ and suppose that $G$ acts by $\sigma(x,h) = (h-x,h)$, which is the reflection swapping its vertices. Since $P^\sigma = \left\{\left(\frac 12, 1\right)\right\}$, we have $\Lambda = \ZZ\cdot(1,2)$, so every $\sigma$-fixed vector has even height and $\hgt(\Lambda) = 2\ZZ$. On the other hand, if some lattice point of $P$ is $\sigma$-fixed, then that point has height one, hence $\hgt(\Lambda) = \ZZ$. This happens for the diagonally reflected square in Example~\ref{ex:swap-square} below where $\sigma$ fixes the vertex $(0,0)$. Examples~\ref{ex:square} and \ref{ex:swap-square} below show the unit square under the action of two different reflections.

\begin{example}[Orthogonally reflected square]\label{ex:square}
Let $n = 2$, write $(x,y,h)$ for the coordinates of the lattice $\ZZ^3$ with $\widetilde M_0 = \{h=0\}$, define the fixed vector $\alpha = (1,0,2)$, and $\sigma$ act by $\sigma(x,y,h) = (h-x,\,y,\,h)$, which is a reflection of the height-one plane in the line $x = \frac12$. Take the polytope $P = [0,1]^2 \times \{1\}$. See the left side of Figure~\ref{fig:squares}. Here, we have $p = 2$, $q = 1$, and the type is $(1,0,1)$. Every fixed vector $(x,y,2x)$ has even height, so the height group is $\hgt(\Lambda) = 2\ZZ$. By Corollary~\ref{cor:poly}(iv), $\varphi[t]$ is not polynomial. Explicitly, we have
\[
\varphi[t](\sigma) = \frac{1+t^2}{1+t} = 1 - t + 2t^2 - 2t^3 + 2t^4 - \dots.
\]
\end{example}

\begin{example}[Diagonally reflected square]\label{ex:swap-square}
Consider the same setup as Example~\ref{ex:square} except with the fixed vector $\alpha = (0,0,1)$ and with $\sigma$ acting by $\sigma(x,y,h) = (y,x,h)$ on the same polytope $P = [0,1]^2\times\{1\}$. The new vector $\alpha$ is $\sigma$-fixed and lies outside $\widetilde M_0$, as the setup of Section~\ref{setup} requires. See the right side of Figure~\ref{fig:squares}. As before, we have $p = 2$, $q = 1$, and type $(1,0,1)$. In this case, the fixed polytope $P^\sigma$ is the lattice segment from $(0,0,1)$ to $(1,1,1)$. By Corollary~\ref{cor:lattice-fixed}, $\varphi[t]$ is effective, and explicitly $\varphi[t] = (1+t)\chi_0$. 
\end{example}

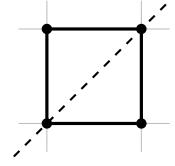
\begin{figure}[t]
\centering
\begin{tikzpicture}[scale=1.25]
  \begin{scope}
    \draw[gray!50] (-0.3,-0.3) grid[step=1] (1.3,1.3);
    \draw[very thick] (0,0) rectangle (1,1);
    \draw[dashed,thick] (0.5,-0.35) -- (0.5,1.35);
    \foreach \p in {(0,0),(1,0),(0,1),(1,1)} {\fill \p circle (0.055);}
    \node at (0.5,-0.75) {\small No lattice point of $P$ lies on $\Fix(\sigma)$};
    \node at (0.5,1.65) {\small Example~\ref{ex:square}};
  \end{scope}
  \begin{scope}[xshift=5cm]
    \draw[gray!50] (-0.3,-0.3) grid[step=1] (1.3,1.3);
    \draw[very thick] (0,0) rectangle (1,1);
    \draw[dashed,thick] (-0.35,-0.35) -- (1.35,1.35);
    \foreach \p in {(0,0),(1,0),(0,1),(1,1)} {\fill \p circle (0.055);}
    \node at (0.5,-0.75) {\small $P^\sigma$ is a lattice segment};
    \node at (0.5,1.65) {\small Example~\ref{ex:swap-square}};
  \end{scope}
\end{tikzpicture}
\caption{Two unit squares in Examples~\ref{ex:square} and \ref{ex:swap-square}. On the left $\varphi[t]$ is not a polynomial. On the right $\varphi[t]$ is effective.}
\label{fig:squares}
\end{figure}

The following straightforward observation gives us an exact criterion for effectiveness.
\begin{proposition}[Effectiveness of {$\varphi[t]$}]\label{prop:eff}
A class $c \in \MR(G)$ is effective if and only if $|c(\sigma)| \le c(e)$. In particular, $\varphi[t]$ is effective if and only if $|\varphi_i(\sigma)| \le (h^*_P)_i$ for all $i$, where $\varphi_i$ is the coefficient of $t^i$ in $\varphi[t]$.
\end{proposition}

\begin{proof}
This result is used in \cite[Section~3]{CHK}. For completeness, we recall the proof.

Let $c$ be a class in $\MR(G)$ and recall that $\chi_0$ and $\chi_1$ are the trivial and sign characters of $G$. Write $c = m\chi_0 + m'\chi_1$ for integers $m$ and $m'$. The class $c$ is effective precisely when both $m$ and $m'$ are nonnegative. Evaluating the class at the two group elements gives $c(e) = m + m'$ and $c(\sigma) = m - m'$, so $m = \frac12(c(e)+c(\sigma))$ and $m' = \frac12(c(e)-c(\sigma))$. Hence both $m$ and $m'$ are nonnegative if and only if $|c(\sigma)| \le c(e)$, which is the first statement.

For the second statement, the series $\varphi[t]$ is effective exactly when each coefficient $\varphi_i$ is effective. The evaluation $\varphi[t](e) = h^*_P(t)$ gives $\varphi_i(e) = (h^*_P)_i$, so applying the first statement to each coefficient $\varphi_i$ shows that $\varphi[t]$ is effective if and only if $|\varphi_i(\sigma)| \le (h^*_P)_i$ for all $i$.
\end{proof}

\medskip
\noindent \textbf{Polynomiality.}
By Lemma~\ref{lem:fixed}(a) the dilate $2P^\sigma$ is a lattice polytope, so by Ehrhart's theorem for rational polytopes \cite[Theorem~4.6.8]{StaEC1} the function $N$ is a quasi-polynomial of period dividing two. So there are polynomials $\Ev, \Od \in \QQ[k]$ of degree at most $p-1$ such that:
\[
N(k) = \#(kP^\sigma \cap \ZZ^{n+1}) = \begin{cases}
    \Ev(k) & \text{if } k \text{ is even}, \\
    \Od(k) & \text{if } k \text{ is odd}.
\end{cases}
\]

\begin{proposition}[Polynomiality of {$\varphi[t]$}]\label{prop:poly}
With the notation above, $\varphi[t]$ is polynomial if and only if
$\deg(\Ev-\Od) \le q-1$,
where we take the degree of the zero polynomial to be $-\infty$. More precisely, the generating function $\Ehr_{P^\sigma}(t)$ has a pole of order $p$ at $t=1$ and no other poles except possibly at $t = -1$ where: its order is $1 + \deg(\Ev-\Od)$ if $\Ev \neq \Od$; and there is no pole if $\Ev = \Od$.
\end{proposition}

Example~\ref{ex:TxI} shows that $\Ev$ and $\Od$ may agree when $P^\sigma$ is not a lattice polytope. In that example the Ehrhart quasi-polynomial of $P^\sigma$ has period strictly smaller than the denominator of $P^\sigma$, a phenomenon called \emph{period collapse} \cite{MW05}; see Remark~\ref{rem:mcmullen}. The next lemma computes $\deg(\Ev-\Od)$ by taking cases on the height group $\hgt(\Lambda)$, which by Lemma~\ref{lem:fixed}(d) is either $\ZZ$ or $2\ZZ$. We will use this lemma to prove Proposition~\ref{prop:poly}.

\begin{lemma}\label{lem:leading}
The polynomial $\Ev$ has degree $p-1$ and positive leading coefficient, and the polynomial $\Od$ has degree at most $p-1$ and nonnegative leading coefficient,
hence $\deg(\Ev+\Od) = p-1$. If $\hgt(\Lambda)=\ZZ$, then $\Ev$ and $\Od$ have the same leading coefficient and satisfy $\deg(\Ev-\Od) \le p-2$. Otherwise, if
$\hgt(\Lambda) = 2\ZZ$, then $\Od$ is zero and $\deg(\Ev-\Od) = p-1$.
\end{lemma}

The following example illustrates Proposition~\ref{prop:poly} and Lemma~\ref{lem:leading} for the two order-two actions on the unit square.

\begin{example}\label{ex:poly-in-action}
Consider the orthogonally reflected square in Example~\ref{ex:square}. The fixed points of $kP$ are $(\frac k2, y, k)$ with $0 \le y \le k$, which are lattice points only for even $k$, so $\Ev(k) = k+1$ and $\Od(k) = 0$. Here, the height group is $2\ZZ$ and $\deg(\Ev-\Od) = 1 = p-1$, following Lemma~\ref{lem:leading}. Since $q - 1$ is zero, Proposition~\ref{prop:poly} tells us that $\varphi[t]$ is not polynomial, and indeed
\[
\Ehr_{P^\sigma}(t) = \sum_{m\ge0}(2m+1)t^{2m} = \frac{1+t^2}{(1-t^2)^2}, \qquad
\varphi[t](\sigma) = (1-t)^2(1+t)\Ehr_{P^\sigma}(t) = \frac{1+t^2}{1+t}.
\]
For the diagonally reflected square in Example~\ref{ex:swap-square}, the fixed polytope $P^\sigma$ is a lattice segment. We have $N(k) = k+1$ for all $k$, so $\Ev(k) = \Od(k) = N(k)$, and the evaluation $\varphi[t](\sigma) = (1-t)^2(1+t)/(1-t)^2 = 1+t$ is a polynomial.
\end{example}

\begin{proof}[Proof of Lemma~\ref{lem:leading}]
We first establish the degrees and the leading coefficients of the two constituents, and then compare those leading coefficients in the two cases for the height group.

By definition $N(k)$ counts the $\sigma$-fixed lattice points of $kP$, so $N(k) = \#(kP^\sigma\cap\Lambda)$.
By Lemma~\ref{lem:fixed}(a) and (c), the dilate $2P^\sigma$ is a lattice polytope of dimension $p-1$, so its vertices lie in $\Lambda$. It follows that $N$ is a quasi-polynomial of period dividing $2$ \cite[Theorem~4.6.8]{StaEC1}, so the polynomials $\Ev$ and $\Od$ are well-defined. Moreover, the function $m \mapsto \Ev(2m) = \#(m\cdot 2P^\sigma \cap \Lambda)$ is the Ehrhart polynomial of $2P^\sigma$, which has degree $p-1$. Thus $\Ev(k)$ has degree $p-1$ and its leading coefficient is the relative volume $\nu$ of the fixed polytope $P^\sigma$ with respect to the sublattice $\Lambda \cap \{\hgt = 0\}$, which is positive. The polynomial $\Od$ has degree at most $p-1$ by its definition above, and its leading coefficient is nonnegative because $\Od(k) = N(k) \ge 0$ for all $k$ odd and large enough. So the sum $\Ev+\Od$ has positive leading coefficient in degree $p-1$, that is, $\deg(\Ev+\Od) = p-1$. That proves the first part of the lemma.

It remains to compare the leading coefficients of $\Ev$ and $\Od$, which we do using the index of the fixed polytope $P^\sigma$. Recall that the index of a rational polytope $Q$ is the least positive integer $m$ such
that the affine span of $mQ$ contains a lattice point. The affine span of $P^\sigma$ lies at height one inside $\Fix(\sigma)$, so the index of
$P^\sigma$ is one if and only if the height group is all of $\ZZ$, and the index is two when $\hgt(\Lambda) = 2\ZZ$. By Lemma~\ref{lem:fixed}(d), these are the only cases for the index.
Let $c(k)$ be the leading coefficient of the Ehrhart quasi-polynomial $N$. Then $c(k)$ is equal to the relative volume $\nu$ of $P^\sigma$ when the index divides $k$ and is zero otherwise. See, for example, \cite[Section~2]{Sta11}.

Suppose first that $\hgt(\Lambda) = \ZZ$. Then the index is one, so both constituents $\Ev$ and $\Od$ have leading coefficient $\nu$. Hence the difference satisfies $\deg(\Ev-\Od) \le p-2$. Suppose instead that $\hgt(\Lambda) = 2\ZZ$. Then every $\sigma$-fixed lattice vector has even height, so the dilate $kP^\sigma$ contains no lattice point for odd $k$ and the polynomial $\Od$ is zero. Hence $\deg(\Ev-\Od) = \deg\Ev = p-1$. So the two cases for the height group give the two remaining claims of the lemma.
\end{proof}

With this description of the degrees and leading coefficients, we prove Proposition~\ref{prop:poly}.

\begin{proof}[Proof of Proposition~\ref{prop:poly}]
Note that the evaluation of $\varphi[t]$ at the identity $e \in G$ is always a polynomial. So $\varphi[t]$ is a polynomial if and only if $\varphi[t](\sigma)$ is a polynomial. We first compute the poles of the generating function $\Ehr_{P^\sigma}(t)$, and apply Equation~\eqref{eq:evals} when $\varphi[t](\sigma)$ is a polynomial.

Observe that $N(k) = \frac12(\Ev+\Od)(k) + \frac12(-1)^k(\Ev-\Od)(k)$, so we have
\[
\Ehr_{P^\sigma}(t) = \sum_{k\ge0}\frac12(\Ev+\Od)(k)t^k
                   + \sum_{k\ge0}\frac12(\Ev-\Od)(k)(-t)^k .
\]
In general, if $f$ is a polynomial of degree $d \ge 0$ with leading coefficient $\lambda$, then its generating function is given by $\sum_{k\ge0}f(k)t^k = g(t)/(1-t)^{d+1}$ where $g$ is a polynomial satisfying $g(1) = d!\cdot\lambda\ne0$. By Lemma~\ref{lem:leading}, the sum $\Ev+\Od$ has degree
$p-1$, so the first sum of $\Ehr_{P^\sigma}(t)$ above is given by $g_+(t)/(1-t)^p$ where $g_+(1)\ne0$.

We now take two cases, according to whether the two constituents agree. Suppose first that $\Ev \neq \Od$, and define $d := \deg(\Ev-\Od)$. Then the second sum of $\Ehr_{P^\sigma}(t)$ is given by $g_-(-t)/(1+t)^{d+1}$ where $g_-(1)\ne0$. So the generating function $\Ehr_{P^\sigma}(t)$ has a pole of order $p$ at $t = 1$ and a pole of order $1+d$ at $t = -1$.
By Equation~\eqref{eq:evals} we have
\[
\varphi[t](\sigma) = (1+t)^q g_+(t) + (1-t)^p(1+t)^{q-d-1}g_-(-t).
\]
The first summand $(1+t)^qg_+(t)$ is a polynomial, so polynomiality of $\varphi[t](\sigma)$ is completely determined by the second summand. Evaluating at $t = -1$ shows that the numerator $(1-t)^pg_-(-t)$ does not vanish there. So the second summand is a polynomial exactly when the exponent $q-d-1$ is nonnegative, that is, exactly when $d \le q-1$.

Suppose instead that $\Ev = \Od$. Then the second sum of $\Ehr_{P^\sigma}(t)$ vanishes, so $\Ehr_{P^\sigma}(t) = g_+(t)/(1-t)^p$ has a pole of order $p$ at $t = 1$ and no pole at $t = -1$. The evaluation is $\varphi[t](\sigma) = (1+t)^qg_+(t)$, which is always polynomial, and the difference $\Ev-\Od$ is the zero polynomial, of degree $-\infty \le q-1$.

In both cases the poles of $\Ehr_{P^\sigma}(t)$ are as described in the second statement, and $\varphi[t]$ is polynomial if and only if $\deg(\Ev-\Od) \le q-1$.
\end{proof}

Using these results, we obtain the following explicit cases for which $\varphi[t]$ is polynomial.

\begin{corollary}[Cases where {$\varphi[t]$} is polynomial]\label{cor:poly}
We have the following.
\begin{enumerate}[label=(\roman*)]
    \item If $p \le q$, then $\varphi[t]$ is polynomial.
    \item If $P^\sigma$ is a lattice polytope then: $N$ is a polynomial, $\varphi[t]$ is polynomial, and the equivariant $H^*$-series evaluated at $\sigma$ is $\varphi[t](\sigma) = (1+t)^q h^*_{P^\sigma}(t)$.
    \item If $p = q+1$, then $\varphi[t]$ is polynomial if and only if $\hgt(\Lambda)=\ZZ$.
    \item If $\hgt(\Lambda)=2\ZZ$, then $\varphi[t]$ is polynomial if and only if $p \le q$.
\end{enumerate}
\end{corollary}

\begin{proof}
(i) Suppose that $p \le q$. By Lemma~\ref{lem:leading}, we have $\deg(\Ev-\Od)\le p-1 \le q-1$, so $\varphi[t]$ is polynomial by Proposition~\ref{prop:poly}.

(ii) Suppose that $P^\sigma$ is a lattice polytope. Then the function $N$ is the Ehrhart polynomial of $P^\sigma$, so $N$ is a polynomial and the two constituents agree, $\Ev = \Od$. By Proposition~\ref{prop:poly}, the series $\varphi[t]$ is polynomial. The Ehrhart series is $\Ehr_{P^\sigma}(t) = h^*_{P^\sigma}(t)/(1-t)^p$, so Equation~\eqref{eq:evals} gives $\varphi[t](\sigma) = (1+t)^qh^*_{P^\sigma}(t)$.

(iii) Let $p = q+1$. If $\hgt(\Lambda) = \ZZ$, then by Lemma~\ref{lem:leading} we have $\deg(\Ev-\Od)\le p-2 = q-1$ and so $\varphi[t]$ is polynomial by Proposition~\ref{prop:poly}. If $\hgt(\Lambda) = 2\ZZ$, then $\deg(\Ev-\Od) = p-1 = q > q-1$ so $\varphi[t]$ is not polynomial by the same proposition. 

(iv) Let $\hgt(\Lambda) = 2\ZZ$. By Lemma~\ref{lem:leading}, we have $\deg(\Ev-\Od) = p-1$. By Proposition~\ref{prop:poly}, it follows that $\varphi[t]$ is polynomial if and only if $p-1 \le q-1$, that is, if and only if $p \le q$.
\end{proof}

\begin{remark}\label{rem:poly-lit}
Let us relate these results with some of those in the literature.
Part (ii) of Corollary~\ref{cor:poly} is a special case of \cite[Lemma~7.1]{Sta11} where Stapledon's proof gives the formula $\varphi[t](g) = h^*_{P^g}(t)\det(I-\rho(g)t)_{(M^g)^\perp}$ which specialises to the formula in the final part of (ii). The non-polynomiality portions of (iii) and (iv) are given by \cite[Lemma~7.3]{Sta11}. In that lemma, the requirement for non-polynomiality is $\dim P^g > (d-r+1)/r$ where $r$ the \emph{index} of the fixed polytope $P^g$: the smallest positive integer $r$ such that the affine span of $rP^g$ contains a lattice point. In our case, this condition is exactly $p > q$. 
Lemma~\ref{lem:leading} uses the standard result that the leading coefficient of the Ehrhart quasi-polynomial is the volume of the polytope when its index divides the dilation factor and is zero otherwise, see for example \cite[Section~2]{Sta11} and \cite[Corollary~5.11 and Remark~5.12]{Sta11}. 
\end{remark}

\begin{remark}\label{rem:dim2}
For lattice polygons under an order-two symmetry, Corollary~\ref{cor:poly} recovers the criterion \cite[Proposition~3.25]{Sta23}. In that more general setting, the equivariant $H^*$ series is non-polynomial if and only if some group element acts, under suitable coordinates, by either a quarter turn about the point $(\frac12,\frac12)$ or a reflection in the line $\{x = \frac12\}$. Since our group has order two, only the second occurs and it is exactly when $\hgt(\Lambda)=2\ZZ$ with $p = 2 > 1 = q$. See Example~\ref{ex:square}.
\end{remark}

\begin{remark}[Odd faces and period collapse]\label{rem:mcmullen}
Call a face $\Gamma$ of $P^\sigma$ \emph{odd} if its affine span $\aff(\Gamma)$ contains no point of the lattice $\ZZ^{n+1}$. Let $k_{\mathrm{odd}}$ be the largest dimension of an odd face. By McMullen's bounds on the periods of the coefficients of an Ehrhart quasi-polynomial \cite{McM78}, the degree bound $\deg(\Ev-\Od) \le k_{\mathrm{odd}}$ holds. So, if $k_{\mathrm{odd}} \le q-1$, then $\varphi[t]$ is polynomial. However, the converse fails, because the contributions of different odd faces may cancel. This phenomenon is period collapse \cite{MW05}, and occurs in Example~\ref{ex:TxI}. Thus it does not seem feasible to obtain a classification of polynomiality of $\varphi[t]$ in terms of the faces of $P^\sigma$.
\end{remark}

\section{Stapledon's hierarchy}\label{sec:known}

In this section, we recall the following chain of conditions, which we call \emph{Stapledon's hierarchy} and examine known cases of the Effectiveness Conjecture.

\begin{theorem}[{\cite[Theorem~1.4]{Sta23}}; $(2)\Rightarrow(4)$ is {\cite[Theorem~7.7]{Sta11}}]\label{thm:stapledon}
Fix the setup of Section~\ref{setup}, with $S_P$ the Ehrhart ring of the polytope $P$, and consider the conditions:
\begin{enumerate}[label=(\arabic*)]
    \item $P$ has a $G$-invariant regular lattice triangulation;
    \item there is a $G$-invariant polynomial $f \in \CC[M]$ that is nondegenerate with respect to $P$;
    \item the ring $S_P$ has a $G$-stable l.s.o.p.\ $V$ whose class satisfies $[V] = [\rho\otimes\CC]$ in the representation ring $\MR(G)$;
    \item $\varphi[t]$ is effective;
    \item $\varphi[t]$ is polynomial.
\end{enumerate}
Then the implications $(1)\Rightarrow(2)\Rightarrow(3)\Rightarrow(4)\Rightarrow(5)$ hold. In addition, if $P$ has a $G$-invariant lattice triangulation (regular or not) then condition (4) is satisfied.
\end{theorem}

If $n = 2$, then all five conditions are equivalent \cite[Proposition~3.25]{Sta23}. The converse $(2)\Rightarrow(1)$ is false \cite[Example~4.35]{Sta23}, and so is $(3)\Rightarrow(2)$ \cite[Examples~4.37--4.39]{Sta23}. The last of those three examples is the counterexample of Santos and Stapledon, see \cite[Theorem~1.2]{EKS}, which disproved the original three-condition form of the conjecture. 

Example~\ref{ex:counterexample} shows that $(5)$ does not imply $(4)$, and Example~\ref{ex:TxI} below shows that $(4)$ does not imply $(3)$. Hence the chain of implications in Theorem~\ref{thm:stapledon} is strict at each step.

\medskip

\noindent \textbf{The toric face ring.}
Stapledon asks about a weakening of condition (3) in which the Ehrhart ring is replaced by the toric face ring of an invariant subdivision. Let us recall its construction. A \emph{lattice polyhedral subdivision} $\mathcal S$ of $P$ is a polyhedral subdivision of $P$ each of whose cells is a lattice polytope, and $\mathcal S$ is \emph{proper} if $P$ is not itself a cell of $\mathcal S$. For a cell $\Gamma$ of $\mathcal S$, write $C_\Gamma = \RR_{\ge0}\Gamma$ for the cone over $\Gamma$. Since the cells of $\mathcal S$ cover $P$, every point of the cone $C_P$ lies in $C_\Gamma$ for some cell $\Gamma$. The \emph{toric face ring} of $\mathcal S$ is the $\CC$-vector space
\[
S_{\mathcal S} \ :=\ \bigoplus_{u \in C_P \cap \ZZ^{n+1}} \CC\,x^u ,
\qquad \text{with grading} \quad \deg x^u = \hgt(u),
\]
equipped with the multiplication
\[
x^u \cdot x^v \ =\
\begin{cases}
x^{u+v} & \text{if } u, v \in C_\Gamma \text{ for some cell } \Gamma \text{ of } \mathcal S,\\
0 & \text{otherwise.}
\end{cases}
\]
Toric face rings are studied in \cite{IR07}, see also \cite[Section~4.5]{Sta23}. The underlying graded vector space of $S_{\mathcal S}$ coincides with the Ehrhart ring $S_P$, and they coincide as rings when $P$ is the only maximal cell of $\mathcal S$. A group acting on $P$ and preserving $\mathcal S$ permutes the cells, hence acts on $S_{\mathcal S}$ by $g\cdot x^u = x^{g(u)}$, which preserves the grading. So the notions of a $G$-stable subspace of $(S_{\mathcal S})_1$ and of an l.s.o.p.\ of $S_{\mathcal S}$ follow similarly as in Section~\ref{setup}.

\begin{question}[{\cite[Question~4.33]{Sta23}}]\label{q:433}
Consider the following weakening of condition (3) of Theorem~\ref{thm:stapledon}:
\begin{itemize}[]
    \item[(3$'$)] there is a $G$-invariant lattice polyhedral subdivision $\mathcal S$ of $P$ whose toric face ring $S_{\mathcal S}$ has a $G$-stable l.s.o.p.\ $V \subseteq (S_{\mathcal S})_1$ with class $[V] = [\rho\otimes\CC]$ in the representation ring $\MR(G)$.
\end{itemize}
Condition (3) implies condition (3$'$), and condition (3$'$) implies condition (4). To what extent are the implications $(3)\Rightarrow(3')\Rightarrow(4)\Rightarrow(5)$ reversible?
\end{question}

We have stated the question in the notation of Theorem~\ref{thm:stapledon}; the two implications recorded in it are \cite[Corollary~4.32]{Sta23}. Stapledon notes that the question has a positive answer in every example of \cite{Sta23} apart from \cite[Example~2.16]{Sta23}, where the answer is not known. Corollary~\ref{cor:q433} shows that condition (4) does not imply condition (3$'$), which settles the implication $(4)\Rightarrow(3')$ of Question~\ref{q:433} in the negative for groups of order two. The reversibility of $(3')\Rightarrow(3)$ remains open.

\medskip

\noindent \textbf{Warning.} 
Stapledon asserts in \cite[Corollary~7.10]{Sta11} that $\varphi[t]$ is effective if every face $F$ of $P$ with $\dim(F) > 1$ contains a $G_F$-fixed lattice point, where the \emph{stabiliser} $G_F$ is the subgroup of elements $g \in G$ with $g(F) = F$. The erratum in \cite{Sta23} reports that the base-point-freeness criterion \cite[Remark~7.8]{Sta11} required by that corollary is false, as \cite[Example~3.23]{Sta23} shows. To solve this problem, we consider the orbits of lattice points as orbits of the whole group instead of the stabiliser. As a consequence, the proof of \cite[Corollary~7.10]{Sta11} is incomplete and, as far as we are aware, no proof or counterexample is known. We note that the reworded version \cite[Corollary~2.10]{EKS} is false by the same example.

\subsection{Known cases}
For a group $G = \langle \sigma \rangle$ of order two, Conjecture~\ref{conj:eff} holds in the following cases:
\begin{itemize}[]
    \item if the dimension $n \le 2$ \cite[Corollary~6.7]{Sta11}, 
    \item if $P$ is a simplex, then the coefficients of $\varphi[t]$ are permutation representations \cite[Proposition~6.1]{Sta11},
    \item if $P = [0,1]^n$ is the unit cube and $\sigma$ acts by permuting its coordinates, then the result is obtained by restricting \cite[Proposition~9.7]{Sta11} from the full symmetric group,
    \item if $\sigma$ has a unique fixed point then the coefficients are permutation representations \cite[Example~2.16]{Sta23}, this result generalises the case in \cite[Corollary~11.1]{Sta11} where the polytope is centrally symmetric about a lattice point,
    \item if $P$ has $n+2$ vertices that affinely generate the affine lattice $M$ \cite[Proposition~4.34]{Sta23},
    \item if $P^\sigma$ is a lattice polytope \cite[Example~3.20]{Sta23},
    \item if $P$ has a $\sigma$-invariant lattice triangulation \cite[Theorem~1.4]{Sta23}, which holds, for instance, for centrally symmetric polytopes \cite[Example~3.21]{Sta23},
    \item if $P$ is the symmetric edge polytope of a cycle \cite{CHK},
    \item if $P$ is a graphic zonotope of an odd path \cite{EKS}.
\end{itemize}

We also note two other cases. D'Al\`i and Delucchi consider polytopes with a unimodular triangulation that is invariant under a translative group action. In this case, the $H^*$-series is effective \cite[Theorem~5.2]{DD}. And D'Al\`i and Higashitani show that the $H^*$-series of order polytopes of graded posets satisfy the stronger condition of $\gamma$-effectiveness \cite[Theorem~4.5]{DH}.

As an aside, if we drop the condition that the vertices are lattice points and instead take them to be rational points, then the conjecture is false. In \cite{CHK}, the authors give a reflectionally symmetric rational polytope whose equivariant $H^*$-series is polynomial but not effective. Example~\ref{ex:counterexample} shows that it is also false for lattice polytopes. 

Note that our counterexample shows, for the case above where $P$ has $n+2$ vertices, that the hypothesis of affine generation cannot be dropped. The fixed polytope of that counterexample is not a lattice polytope, so the hypothesis of \cite[Example~3.20]{Sta23} cannot be dropped either.

\section{Criterion for a linear system of parameters}\label{sec:main}

The main result of this section is Theorem~\ref{thm:lsop}, which determines when the Ehrhart ring $S_P$ has a $\sigma$-stable linear system of parameters of the correct representation type, i.e., it is equal to $[\rho \otimes \CC] \in \MR(G)$ where $\rho : G \rightarrow \GL_{n+1}(\ZZ)$ is the natural inclusion. The condition is on the faces of $P$, and is given in terms of the following notion.

\begin{definition}[$\sigma$-odd and obstructing faces]\label{def:odd}
Let $F$ be a face of $P$ and $M_F \subseteq \ZZ^{n+1}$ the subgroup generated by the lattice points of $F$. We say that $F$ is \emph{$\sigma$-odd} if $\sigma(F) = F$ and every $\sigma$-fixed element of the group $M_F$ has even height, i.e., $\hgt(M_F \cap \Fix(\sigma)) \subseteq 2\ZZ$. Note that a $\sigma$-odd face contains no $\sigma$-fixed lattice point. We call a $\sigma$-odd face $F$ \emph{obstructing} if $\dim F^\sigma \ge q$, where $q$ is the dimension of the $(-1)$-eigenspace of $\sigma$ as in Section~\ref{sec:order-two}. Finally, we say that $P$ satisfies \emph{(Gen)} if its lattice points affinely generate the affine lattice $M$, equivalently if $M_P$ is all of $\ZZ^{n+1}$.
\end{definition}

The simplest kind of $\sigma$-odd face is an edge $F$ whose vertices $u$ and $v$ are swapped by $\sigma$ and contains no lattice point in its relative interior. In that case $M_F$ is freely generated by $u$ and $v$, the $\sigma$-fixed elements of $M_F$ are the multiples of $u+v$, and each such multiple has even height. On the other hand, a face fixed pointwise by $\sigma$ is never $\sigma$-odd as its vertices are $\sigma$-fixed height one lattice points.

For obstructing faces, let us compare Examples~\ref{ex:bipyramid} and \ref{ex:pyramid}, which both have $q = 1$. In Example~\ref{ex:bipyramid}, the two $\sigma$-odd edges $F$ each have $\dim F^\sigma = 0 < q$, so neither is obstructing and so $P$ has no obstructing face. In Example~\ref{ex:pyramid}, the $\sigma$-odd facet $F$ has $\dim F^\sigma = 1 \ge q$ and is therefore obstructing. Note that in the definition of obstructing, we may take $F = P$. For instance, in Example~\ref{ex:square}, the polytope is itself $\sigma$-odd, with $\dim P^\sigma = 1 \ge q$.

Every lattice polygon satisfies (Gen), because every lattice polygon has a unimodular triangulation and so its lattice points affinely generate the affine lattice $M$. The polytope of Example~\ref{ex:circuit} satisfies (Gen) as well. The simplest example where (Gen) fails is the Reeve simplex $\conv\{(0,0,0),(1,0,0),(0,1,0),(1,1,2)\}$, which is a lattice simplex with no non-vertex lattice points, but whose vertices generate a sublattice of index two.

We are now ready to state our main result, which we will prove in Section~\ref{sec:proofs}.

\begin{theorem}[Criterion for an invariant l.s.o.p.]\label{thm:lsop}
Let $G = \langle\sigma\rangle \subseteq \GL_{n+1}(\ZZ)$ be a group of order two acting on $P$ as in Section~\ref{sec:complete_setup}. Recall that $S_P$ is the Ehrhart ring of $P$ and $p$ and $q$ are the dimensions of the $(\pm1)$-eigenspaces of $\sigma$. Then $S_P$ has a $\sigma$-stable l.s.o.p.\ $V \subseteq (S_P)_1$ of class $[V] = [\rho\otimes\CC] = p\chi_0 + q\chi_1$ if and only if $P$ has no obstructing face. In that case $\varphi[t]$ is the graded character of $S_P/(V)$, and in particular it is effective.
\end{theorem}

\begin{remark}\label{rem:cor710}
Theorem~\ref{thm:lsop} proves \cite[Corollary~7.10]{Sta11} for the case of a group of order two, whose published proof is incomplete. See the warning in Section~\ref{sec:known}. Assume that every face of $P$ of dimension greater than one contains a lattice point fixed by its stabiliser. We check that no face of $P$ is obstructing, taking cases on the dimension of the face. A face that is not $\sigma$-invariant is not $\sigma$-odd, by Definition~\ref{def:odd}, so we may restrict to the $\sigma$-invariant faces, whose stabiliser is all of $G$. A $\sigma$-invariant face of dimension greater than one contains a $\sigma$-fixed lattice point by assumption, so such a face is not $\sigma$-odd. A $\sigma$-invariant vertex is itself a $\sigma$-fixed lattice point of height one, so a vertex is not $\sigma$-odd either. Finally, a $\sigma$-invariant edge $F$ is either fixed pointwise, in which case the endpoints of $F$ are $\sigma$-fixed lattice points of height one and $F$ is not $\sigma$-odd, or has its endpoints swapped, in which case $\dim F^\sigma = 0 < q$ and the edge $F$ is not obstructing. Hence $P$ has no obstructing face, and $\varphi[t]$ is effective by Theorem~\ref{thm:lsop}.
\end{remark}

By Lemma~\ref{lem:fixed}(c) a proper face satisfies $\dim F^\sigma \le p-2$, so a proper face can be obstructing only when $p \ge q+2$. When $p = q+1$ the only face that can be obstructing is $P$ itself, which has $\dim P^\sigma = p-1 = q$. So we obtain the following three corollaries.

\begin{corollary}\label{cor:p-le-q}
If $p \le q$, then $P$ has no obstructing face, and $\varphi[t]$ is effective and polynomial.
\end{corollary}

\begin{proof}
Suppose that $p \le q$. By Lemma~\ref{lem:fixed}(c), every invariant face $F$ satisfies $\dim F^\sigma \le p-1 \le q-1$, so no face of $P$ is obstructing. Since $p \le q$, Corollary~\ref{cor:poly}(i) gives that $\varphi[t]$ is polynomial, and Theorem~\ref{thm:lsop} gives that $\varphi[t]$ is effective.
\end{proof}

In the case where $p = 1$, the above gives the effectiveness part of \cite[Example~2.16]{Sta23} for a group of order two. Moreover, in that case, it is possible to show that the coefficients are permutation representations.

\begin{corollary}[Lattice fixed polytope]\label{cor:lattice-fixed}
If the fixed polytope $P^\sigma$ is a lattice polytope, for instance if the matrix $U$ does not appear in the normal form of $\sigma$, as in Section~\ref{sec:order-two}, then $P$ has no $\sigma$-odd faces and $\varphi[t]$ is effective. In that case, effectiveness of $\varphi[t]$ is equivalent to the coefficient-wise inequality
\[
\left[h^*_P(t)\right]_i \ \ge\ \left[(1+t)^q\,h^*_{P^\sigma}(t)\right]_i \qquad \text{for each } i,
\]
where $[f(t)]_i$ denotes the coefficient of $t^i$ in the polynomial $f$.
\end{corollary}

\begin{proof}
If the matrix $U$ does not appear in the normal form of $\sigma$, then $c = 0$ and the fixed polytope $P^\sigma$ is a lattice polytope by Lemma~\ref{lem:fixed}(a). So suppose from now on that the fixed polytope $P^\sigma$ is a lattice polytope.

Let $F$ be an invariant face of $P$. By Lemma~\ref{lem:fixed}(b), the fixed part $F^\sigma$ is a nonempty face of the lattice polytope $P^\sigma$, so $F^\sigma$ has a vertex $u$ in $\ZZ^{n+1}$. That vertex is a $\sigma$-fixed lattice point of $F$ at height one, so the face $F$ is not $\sigma$-odd. Hence $P$ has no $\sigma$-odd face. In particular no face of $P$ is obstructing, so $\varphi[t]$ is effective by Theorem~\ref{thm:lsop}.

For the inequality, Corollary~\ref{cor:poly}(ii) gives $\varphi[t](\sigma) = (1+t)^qh^*_{P^\sigma}(t)$, whose coefficients are nonnegative. So the absolute values in Proposition~\ref{prop:eff} may be dropped, and effectiveness of $\varphi[t]$ is equivalent to the coefficient-wise inequality $\left[h^*_P(t)\right]_i \ge \left[(1+t)^qh^*_{P^\sigma}(t)\right]_i$ for each $i$.
\end{proof}

For a centrally symmetric polytope with central lattice point and $\sigma$ having $p = 1$ and $q = n$, the inequality becomes $h^*_i \ge \binom ni$, which coincides with the bound of Bey, Henk and Wills \cite[Remark~1.6]{BHW}, which also follows from \cite{Sta87} together with \cite{BM85}. The equivalence between the bound and effectiveness in this case is \cite[Section~11]{Sta11}.

We note that the effectiveness half of Corollary~\ref{cor:lattice-fixed} is not new. It appears in \cite[Example~3.20]{Sta23} and is proved using triangulations. Taking $P^\sigma$ as the subpolytope $Q$ in \cite[Corollary~3.18]{Sta23} and $K$ the subcomplex of faces of $P$ disjoint from $Q$, the action on $K$ is translative by \cite[Example~3.15]{Sta23}. So, by \cite[Corollary~3.17]{Sta23}, we obtain an invariant triangulation of $K$. Note, any lattice triangulation of $Q$ is invariant because the involution fixes it pointwise. So the gluing of these two triangulations gives a $\sigma$-invariant regular lattice triangulation of $P$, and effectiveness follows from Theorem~\ref{thm:stapledon}. For a centrally symmetric polytope, the same construction is carried out in \cite[Example~3.21]{Sta23}. Our proof uses the absence of obstructing faces and avoids triangulations. In addition, we obtain the new coefficient-wise inequality above.

\begin{corollary}\label{cor:p-eq-q-plus-1}
Suppose $p = q+1$. Then $P$ has no obstructing face if and only if $P$ is not
$\sigma$-odd, and $\varphi[t]$ is polynomial if and only if the height group
$\hgt(\Lambda)$ is all of $\ZZ$. Suppose moreover that $P$ satisfies (Gen).
Then $\varphi[t]$ is polynomial if and only if it is effective. Hence
$\varphi[t]$ is polynomial and effective if and only if some $\sigma$-fixed
vector of $\ZZ^{n+1}$ has odd height.
\end{corollary}

\begin{proof}
Suppose that $p = q+1$, and let $F$ be a proper face of $P$. By Lemma~\ref{lem:fixed}(c), we have $\dim F^\sigma\le p-2 = q-1$ and $\dim P^\sigma = p-1 = q$. So $P$ has no proper obstructing face, and so $P$ has an obstructing face if and only if $P$ is $\sigma$-odd.
By Corollary~\ref{cor:poly}(iii), we have $\varphi[t]$ is polynomial if and only if $\hgt(\Lambda) = \ZZ$.

Suppose that $P$ satisfies (Gen). Then $M_P$ is all of $\ZZ^{n+1}$ and
$M_P\cap\Fix(\sigma) = \Lambda$. Then $P$ is $\sigma$-odd if and only if every
$\sigma$-fixed vector of $\ZZ^{n+1}$ has even height, that is, if and only if the height group
is $\hgt(\Lambda) = 2\ZZ$. By combining the equivalence just established with the first two claims of the corollary, we deduce that the following are equivalent: the polytope has no obstructing face; $P$ is not $\sigma$-odd; the height group is
$\ZZ$; and $\varphi[t]$ is a polynomial. If $\varphi[t]$ is polynomial, then $P$ has
no obstructing face and so $\varphi[t]$ is effective by Theorem~\ref{thm:lsop}.
Conversely, if $\varphi[t]$ is effective, then $\varphi[t]$ is a polynomial \cite[Remark~2.13]{Sta23}. Hence, under (Gen), the series $\varphi[t]$ is polynomial if and only if $\varphi[t]$ is effective.

For the last statement, there exists a $\sigma$-fixed vector in $\ZZ^{n+1}$ with odd height if and only if $\hgt(\Lambda) = \ZZ$, which by the above happens if and only if $\varphi[t]$ is polynomial and effective.
\end{proof}

This corollary is, in some sense, sharp. Call the equivalence between polynomiality and effectiveness the \emph{first} claim of the corollary, and the criterion in terms of a $\sigma$-fixed vector of odd height the \emph{second} claim. Neither hypothesis ($p = q+1$ or (Gen)) can be dropped from the second claim.
For $p \ge q+2$ the second claim fails. 
The polytope in Example~\ref{ex:pyramid} satisfies (Gen) and has a $\sigma$-fixed vector of odd
height in $\ZZ^4$, but $\varphi[t]$ is not a polynomial. 
For $p \le q$, the line segment $P = [0,1]\times\{1\}$ with $\sigma(x,h) = (h-x,h)$
satisfies (Gen), and $\varphi[t] = \chi_0$ is polynomial and effective, but
every $\sigma$-fixed vector of $\ZZ^2$ has even height. On the other hand, Example~\ref{ex:counterexample} has $p = q+1$ and fails (Gen), and $\varphi[t]$ is polynomial but not effective. So (Gen) cannot be dropped from the first claim either. Whether the hypothesis $p = q+1$ can be dropped from the first claim is open, because neither counterexample of this paper satisfies (Gen); see Remark~\ref{rem:hyp}. In Example~\ref{ex:counterexample}, the polytope has $n+2$ vertices, so it also gives sharpness of the following criterion of Stapledon.

\begin{corollary}[Sharpness of {\cite[Proposition~4.34]{Sta23}}]\label{cor:434-sharp}
The hypothesis of affine generation cannot be dropped from \cite[Proposition~4.34]{Sta23}: there is an $n$-dimensional $\sigma$-invariant lattice polytope with $n+2$ vertices whose equivariant $H^*$-series is polynomial but not effective.
\end{corollary}

\begin{proof}
Example~\ref{ex:counterexample} has $n = 6$ and eight vertices, which are its only lattice points, so the polytope of that example has $n+2$ vertices. Those eight lattice points generate a subgroup of index two, so the polytope does not satisfy (Gen) and the hypothesis of affine generation fails. The equivariant $H^*$-series of that polytope is polynomial but not effective, which gives the required example.
\end{proof}

\begin{table}[t]
\centering
\renewcommand{\arraystretch}{1.15}
\begin{tabular}{>{\raggedright\arraybackslash}p{8.0cm}>{\raggedright\arraybackslash}p{5.2cm}>{\raggedright\arraybackslash}p{3.0cm}}
\toprule
Case & Status & Reference \\
\addlinespace
\midrule
dimension $n \le 2$ & holds & \cite{Sta11,Sta23} \\
\addlinespace
$P$ is a simplex & holds & \cite{Sta11} \\
\addlinespace
$P$ has $n+2$ vertices and (Gen) & holds & \cite{Sta23} \\
\addlinespace
$P$ has invariant lattice triangulation & holds & \cite{Sta23} \\
\addlinespace
eigenspaces with $p \le q$ & holds (always polynomial) & Corollary~\ref{cor:p-le-q} \\
\addlinespace
fixed polytope $P^\sigma$ a lattice polytope & holds & \cite{Sta23}, Corollary~\ref{cor:lattice-fixed} \\
\addlinespace
no obstructing face & holds & Theorem~\ref{thm:lsop} \\
\addlinespace
eigenspaces with $p = q+1$ and (Gen) & holds & Corollary~\ref{cor:p-eq-q-plus-1} \\
\addlinespace
\midrule
eigenspaces with $p = q+1$, polytope $\sigma$-odd, (Gen) does not hold & fails for $n = 6$ & Example~\ref{ex:counterexample} \\
\addlinespace
eigenspaces with $p \ge q+2$, some obstructing face, series polynomial & fails for $n = 6$ & Example~\ref{ex:counterexample-c2} \\
\addlinespace
\bottomrule
\end{tabular}
\caption{Status of Conjecture~\ref{conj:eff} for a group of order two.}
\label{tab:status}
\end{table}

\begin{remark}[Polytopes at other heights]\label{rem:rescale}
Suppose $P$ lies in a hyperplane $\{\ell = k\}$ for a $G$-invariant functional $\ell$ on $\ZZ^{n+1}$ and integer $k > 1$. For instance, the hypersimplex $\Delta(k,n+1) \subseteq \{\sum_i x_i = k\}$ lies at height $k$. If we replace $\ZZ^{n+1}$ by the $G$-invariant sublattice $\{v \in \ZZ^{n+1} : k \mid \ell(v)\}$ and choose a basis, then there is no change to any lattice-point count. To see this, note that the functional $\ell$ takes the value $mk$ on the dilate $mP$, and the polytope now lies at height one. The values $p$ and $q$ are unaffected by the change despite the possible change in the type $(a,b,c)$.
\end{remark}

\begin{remark}[Hypersimplices and permutahedra]\label{rem:hyp}
Let the permutation $\sigma \in S_{n+1}$ be a product of $c'$ disjoint transpositions with $f$ fixed points, acting by permuting coordinates, so that $p = f + c'$ and $q = c'$. Consider the hypersimplex $\Delta(k,n+1)$ with $1 \le k \le n$, rescaled as in Remark~\ref{rem:rescale}. 
Note that the hypersimplex satisfies (Gen), as the differences of its vertices generate all vectors of the form $e_i - e_j$, and those vectors generate the sublattice $M_0 = \{v \in \ZZ^{n+1} \colon \sum_i v_i = 0\}$, the $A_n$ root lattice.
If $f=0$ then $\varphi[t]$ is effective by Corollary~\ref{cor:p-le-q}. For $f = 1$, there is a $\sigma$-fixed coordinate vector $k e_i$ of height one for some $i$, so $\varphi[t]$ is effective by Corollary~\ref{cor:p-eq-q-plus-1}.
If instead $f \ge 2$, then $p \ge q+2$, which lies in case (C2) of Proposition~\ref{prop:open}. We show in Example~\ref{ex:counterexample-c2} that the Effectiveness Conjecture does not hold in that case, so the case of
hypersimplices with $f \ge 2$ is not settled by the results above. However, by the above, hypersimplices
satisfy (Gen), which is a difference from our counterexample.

Note that the equivariant Ehrhart theory of $\Delta(k,n+1)$ under the full symmetric group is studied in \cite{Hyp}, where it is also conjectured that the standard lattice triangulation is invariant under a dihedral subgroup \cite[Conjecture~6.3]{Hyp}. For involutions of that dihedral group, the conjectured invariance would give effectiveness by Theorem~\ref{thm:stapledon}. The same arguments apply to the permutahedron, and by \cite[Lemma~5.2]{ASVM} the evaluation $\varphi[t](\sigma)$ is a polynomial exactly when the number of even parts of the cycle type of the permutation $\sigma$ is: zero; one less than the number of parts; or equal to the number of parts. For us, this means $f \le 1$. Hence Conjecture~\ref{conj:eff} holds for each order-two subgroup of the symmetric group $S_{n+1}$ acting on the permutahedron.
\end{remark}

\section{Examples and counterexamples}\label{sec:examples}

Throughout this section, we assume the setup in Section~\ref{sec:complete_setup}, including the involution $\sigma$, the values $p$ and $q$, the fixed lattice $\Lambda$ and the fixed polytope $P^\sigma$. We also recall the notion of $\sigma$-odd and obstructing faces as well as the condition (Gen) from Definition~\ref{def:odd}. 
The first four examples illustrate the definitions and the criteria, and the last two are counterexamples to Conjecture~\ref{conj:eff}.

\begin{example}[Pyramid]\label{ex:pyramid}
Let $n = 3$, write $(x,y,z,h)$ for the coordinates of $\ZZ^4$ with $\widetilde M_0 = \{h=0\}$ the orthogonal space to the fixed vector $\alpha = (0,0,0,1)$. Let $\sigma$ act by swapping the coordinates $x$ and $y$. Let
\[
P = \conv\{0,\, e_x,\, e_y,\, e_x+e_z,\, e_y+e_z\} \times \{1\}
\]
be a polytope with normalised volume $2$. Its $h^*$-polynomial is $h^*_P = 1+t$. In this case, we have $p = 3$, $q = 1$, and the type is $(2,0,1)$. The facet $F = \conv\{e_x,e_y,e_x+e_z,e_y+e_z\}$ is $\sigma$-odd with $\dim F^\sigma = 1 > q-1$, so $F$ is obstructing. The fixed-point count is $N(k) = (\lfloor k/2\rfloor + 1)^2$ and so
\[
\varphi[t](\sigma) = \frac{1+t^2}{1+t}
\]
is not polynomial so the Effectiveness Conjecture trivially holds. Note that $F$ is a copy of the orthogonally reflected square in Example~\ref{ex:square}. The fixed triangle is drawn in Figure~\ref{fig:pyramid}.

The polytope $P$ is the smallest polytope with an obstructing proper face. In dimension $n \le 2$ no proper face can be obstructing: by Lemma~\ref{lem:fixed}(c) a proper obstructing face forces $p \ge q+2$, hence $n+1 = p+q \ge 2q+2 \ge 4$. In dimension $n = 3$, the only type with $p \ge q+2$ and the normal form of $\sigma$ containing a copy of the matrix $U$ is $(2,0,1)$. An obstructing proper face must be a $\sigma$-odd $2$-face. Such a face has no vertex on the line of symmetry, hence it has at least four vertices and a normalised area of at least $2$. Thus the polytope has at least five vertices and normalised volume at least $2$. The polytope $P$ also has $n+2$ vertices that affinely generate $M$ and admits no $\sigma$-invariant lattice triangulation, so \cite[Proposition~4.34]{Sta23} applies. The Effectiveness Conjecture holds for $P$ for the trivial reason that $\varphi[t]$ is not a polynomial.
\end{example}

\begin{figure}[t]
\centering
\begin{tikzpicture}[scale=1.5]
  \draw[gray!50] (-0.3,-0.3) grid[step=1] (1.3,1.3);
  \draw[very thick] (0,0) -- (0.5,0) -- (0.5,1) -- cycle;
  \draw[dashed,thick] (0.5,-0.35) -- (0.5,1.35);
  \foreach \p in {(0,0),(1,0),(0,1),(1,1)} {\fill \p circle (0.05);}
  \fill[white] (0.5,0) circle (0.05); \draw (0.5,0) circle (0.05);
  \fill[white] (0.5,1) circle (0.05); \draw (0.5,1) circle (0.05);
  \node[anchor=west] at (1.7,0.5) {\small \begin{tabular}{l} The fixed triangle $P^\sigma$, in\\ coordinates $(s,r)$. The edge\\ with $s = \frac12$ is odd\end{tabular}};
  \node at (0.05,-0.2) {\small $0$};
\end{tikzpicture}
\caption{The fixed polytope of Example~\ref{ex:pyramid}. A fixed point of $kP$ is a point $(s,s,r)$ with $0 \le r \le 2s \le k$. Its half-integral (unfilled) vertices lie on the $s = \frac12$, which has no lattice points.}
\label{fig:pyramid}
\end{figure}

\begin{example}[Bipyramid {\cite[Example~3.26]{Sta23}}]\label{ex:bipyramid}
Consider the same setup as Example~\ref{ex:pyramid} except that $\alpha = (0,0,1,1)$ and $\sigma$ acts by $\sigma(x,\,y,\,z,\,h) = (-x-z+h,\,y,\,z,\,h)$. As before, the type is $(2,0,1)$, so $p = 3$ and $q = 1$. In this height-one affine subspace, $\sigma$ is a reflection in the line $\{2x+z=1\}$. Let $Q = [0,1]^2\times\{0\}$ be a square and define
\[
P = \conv\{Q,\ a,\ b\} \times \{1\}, \quad \text{where} \quad a = (0,0,1),\quad b = (1,1,-1).
\]
The polytope $P$ is a bipyramid over the square $Q$ with $h^*$-polynomial $h^*_P = 1+2t+t^2$. Both points $a$ and $b$ are $\sigma$-fixed lattice points, and $P^\sigma$ is the quadrilateral with vertices $a$, $b$, $(\frac12,0,0)$ and $(\frac12,1,0)$, which is a non-lattice rational polytope. The two edges of the square $Q$ that cross $\Fix(\sigma)$ are $\sigma$-odd faces of $P$, but each such face $F$ has $\dim F^\sigma = 0 = q-1$, while every $\sigma$-invariant $2$-face contains one of $a$ and $b$. The polytope $P$ itself is not $\sigma$-odd, because the $\sigma$-fixed lattice point $a$ of $P$ has height one. So $P$ has no obstructing face and $\varphi[t]$ is effective. Indeed, the series is $\varphi[t] = 1 + \chi_{\mathrm{reg}}t + t^2$ where $\chi_{\mathrm{reg}} = \chi_0+\chi_1$ is the regular character \cite[Example~3.26]{Sta23}, so its evaluation is $\varphi[t](\sigma) = 1+t^2$. Here $\Ev(k) = \frac12k^2+k+1$ and $\Od(k) = \frac12k^2+k+\frac12$, so their difference is the constant $\frac12$. Note that in Proposition~\ref{prop:poly}, the polynomiality of $\varphi[t]$ does not require the two constituents to be equal, only that $\deg(\Ev-\Od)\le q-1$. Note that $Q$ is $\sigma$-odd with $\dim Q^\sigma = 1$, but it is a cross-section of $P$ and not a face, so it is not an obstructing face. 
\end{example}

\begin{example}[Effective and no invariant l.s.o.p.]\label{ex:TxI}
Let $n = 4$ and fix coordinates for $\ZZ^5$ given by $(x,y_1,y_2,z,h)$. Let $\widetilde M_0 = \{h = 0\}$ be the space orthogonal to $\alpha = (0,0,0,0,1)$, and let $\sigma$ act by swapping coordinates $y_1$ and $y_2$. Then $p = 4$ and $q = 1$, and the action has type $(3,0,1)$. Let
\[
T = \conv\{(0,0,0),\ (2,0,0),\ (1,1,0),\ (1,0,1)\}, \qquad P = T \times [0,1] \times \{1\}.
\]
Then we have:
\begin{itemize}[]
    \item the $h^*$-polynomial is $h^*_P(t) = 1 + 5t + 2t^2$,
    \item the quasi-polynomial $N(k) = \frac12(k+1)^2(k+2)$ is a polynomial, so $\Ev = \Od = N$, and the evaluation of the equivariant $H^*$-series is $\varphi[t](\sigma) = (1+t)(1+2t)$,
    \item the $H^*$-series is effective and given by $\varphi[t] = (1+4t+2t^2)\chi_0 + t\,\chi_1$,
    \item the square $2$-face $F = [(1,1,0),(1,0,1)]\times[0,1]\times\{1\}$ is $\sigma$-odd of dimension $\dim F^\sigma = 1 \ge q = 1$, and so $F$ is obstructing.
\end{itemize}
By Theorem~\ref{thm:lsop}, $P$ satisfies condition (4: $\varphi[t]$ is effective) but not condition (3: $S_P$ has a $\sigma$-stable l.s.o.p.\ of class $[\rho \otimes \CC]$) of Stapledon's hierarchy (Theorem~\ref{thm:stapledon}). So $P$ has no $\sigma$-invariant regular lattice triangulation and no invariant nondegenerate polynomial. The last conclusion also follows from \cite[Lemma~4.29]{Sta23}, since the face $F$ is a unit square with two disjoint edges that have their lattice points in a single $\sigma$-orbit. By Theorem~\ref{thm:lsop}, this example shows that effectiveness does not imply the existence of an l.s.o.p.\ of $S_P$ of class $[\rho\otimes\CC]$. We recall Question~\ref{q:433}, which asks to what extent the conditions in that result are reversible. A remaining part of Question~\ref{q:433} is whether effectiveness implies condition (3$'$), the existence of a $\sigma$-stable l.s.o.p.\ for the toric face ring of a $\sigma$-invariant lattice polyhedral subdivision of $P$. Corollary~\ref{cor:q433} settles that part with the same example. The fixed polytope is $P^\sigma \cong Q_0 \times [0,1]$, where the triangle $Q_0 = \conv\{(0,0),(2,0),(1,\frac12)\}$ shown in Figure~\ref{fig:TxI} has an Ehrhart quasi-polynomial that undergoes period collapse \cite{MW05}. The vertical edge over the non-lattice point of $Q_0$ is an odd face of $P^\sigma$, so the invariant of Remark~\ref{rem:mcmullen} satisfies $k_{\mathrm{odd}} = 1 > q-1$ even though $\varphi[t]$ is polynomial.
\end{example}

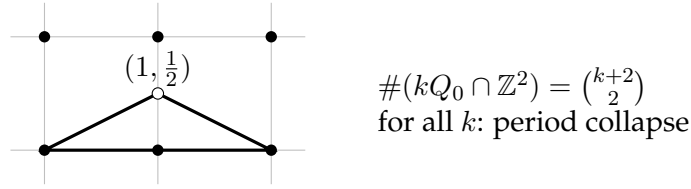
\begin{figure}[t]
\centering
\begin{tikzpicture}[scale=1.5]
  \draw[gray!50] (-0.3,-0.3) grid[step=1] (2.3,1.3);
  \draw[very thick] (0,0) -- (2,0) -- (1,0.5) -- cycle;
  \foreach \p in {(0,0),(1,0),(2,0),(0,1),(1,1),(2,1)} {\fill \p circle (0.05);}
  \fill[white] (1,0.5) circle (0.05); \draw (1,0.5) circle (0.05);
  \node at (1,0.72) {\small $(1,\frac12)$};
  \node[anchor=west] at (2.7,0.4) {\small \begin{tabular}{l} $\#(kQ_0 \cap \ZZ^2) = \binom{k+2}{2}$\\ for all $k$: period collapse \end{tabular}};
\end{tikzpicture}
\caption{The fixed polygon of Example~\ref{ex:TxI}. Its apex is half-integral, yet its Ehrhart counting function is a polynomial, which is what makes $\varphi[t]$ polynomial although $P$ has an obstructing face.}
\label{fig:TxI}
\end{figure}

\begin{remark}[Poles and {$\sigma$}-odd faces]\label{rem:nopole}
Example~\ref{ex:TxI} also shows that the presence of a $\sigma$-odd face $F$ does not force a pole of order $\dim F^\sigma + 1$ at $t = -1$ in the generating function $\Ehr_{P^\sigma}$. On the other hand, in \cite[Examples~4.37--4.39]{Sta23}, a linear system of parameters of the correct class exists but there is no invariant nondegenerate polynomial. In those examples, the groups are cyclic groups of order three and five and the Klein four-group. So, those examples do not apply to the cases treated in this paper.
\end{remark}

\begin{example}\label{ex:circuit}
An example for $p = q+1$ of Corollary~\ref{cor:p-eq-q-plus-1} is illustrated by
\cite[Example~4.35]{Sta23}. Let $n = 4$, and write $(u_1,u_2,v_1,v_2,h)$ as coordinates for $\ZZ^5$ with $\widetilde M_0 = \{h = 0\}$ and vector
$\alpha = (0,0,0,0,1)$. Let $\sigma$ act by exchanging the two blocks of two coordinates:
\[
\sigma(u,v,h) = (v,u,h), \qquad \text{for } u,v \in \RR^2 .
\]
The type of $\sigma$ is $(1,0,2)$, so $p = 3$, $q = 2$ and so $p = q+1$. Let $P$ be the convex
hull of the three $\sigma$-orbits
\[
\begin{array}{ll@{\qquad}ll}
e_1 = (1,0,0,0,1), & f_1 = (0,0,1,0,1), &
e_2 = (0,1,0,0,1), & f_2 = (0,0,0,1,1), \\[2pt]
e_3 = (0,0,1,1,1), & f_3 = (1,1,0,0,1). &&
\end{array}
\]
The six points above affinely generate $M$, so $P$ satisfies (Gen), and they are the only lattice points of $P$. The fixed vector $\alpha = (0,0,0,0,1)$ has height one, so the height group is $\hgt(\Lambda) = \ZZ$. Since (Gen) holds, the group $M_P$ is all of $\ZZ^5$ and so $P$ is not $\sigma$-odd. Hence, by Corollary~\ref{cor:p-eq-q-plus-1}, the series $\varphi[t]$ is both polynomial and effective, and indeed $\varphi[t] = 1 + \chi_1 t + t^2$, with $h^*_P = 1+t+t^2$ and normalised volume $\Vol(P) = 3$. The six lattice points form a circuit in $\RR^5$ with relation $e_1+e_2+e_3 = f_1+f_2+f_3$, so $P$ has exactly two lattice triangulations. The involution $\sigma$ exchanges $e_i$ and $f_i$, hence exchanges those two triangulations, so $P$ admits no $\sigma$-invariant lattice triangulation. As explained in \cite[Example~4.35]{Sta23}, this example is a counterexample to \cite[Conjecture~12.4]{Sta11}, since the linear coefficient is nonzero and effective
but contains no copy of the trivial character.
\end{example}

The next example is the counterexample to Conjecture~\ref{conj:eff}. The setup is similar to Example~\ref{ex:circuit}. In both cases $p = q+1$ and have $n+2$ vertices. The main difference is that in the following example (Gen) does not hold, i.e., the vertices of $P$ do not generate the lattice.

\begin{example}[Counterexample with $p = q+1$]\label{ex:counterexample}
Let $n = 6$, and write the coordinates of $\ZZ^7$ as $(u_1, u_2, u_3, v_1, v_2, v_3, h)$ with $\widetilde M_0 = \{h = 0\}$ and vector $\alpha = (0,0,0,0,0,0,1)$. Let $\sigma$ act by exchanging two blocks of three coordinates:
\[
\sigma(u,v,h) = (v,u,h), \qquad \text{for } u,v \in \RR^3 .
\]
The type of $\sigma$ is $(1,0,3)$, so $p = 4$, $q = 3$ and $p = q+1$. Let $P \subseteq \widetilde M$ be the convex hull of the four $\sigma$-orbits
\[
\begin{array}{ll@{\qquad}ll}
w_1 = (1,0,0,0,0,0), & \sigma (w_1) = (0,0,0,1,0,0), &
w_2 = (0,0,1,0,0,0), & \sigma (w_2) = (0,0,0,0,0,1), \\[2pt]
w_3 = (0,1,0,1,1,0), & \sigma (w_3) = (1,1,0,0,1,0), &
w_4 = (0,0,1,1,1,0), & \sigma (w_4) = (1,1,0,0,0,1),
\end{array}
\]
where each point is put at height one, i.e., setting the last coordinate $h = 1$.
Then we have:
\begin{itemize}[]
    \item $P$ is $6$-dimensional with 
    normalised volume $\Vol(P) = 4$, and its eight vertices are its only lattice points, so it has $n+2$ vertices,
    \item the group $M_P$ has index two in $\ZZ^7$, so (Gen) fails, and every $\sigma$-fixed element of $M_P$ has even height, so $P$ is $\sigma$-odd and is therefore obstructing because $\dim P^\sigma = 3 = q$,
    
    \item the fixed polytope $P^\sigma$ is a simplex, see Figure~\ref{fig:cex}, its four vertices
    are half-integral, and $P^\sigma$ contains no lattice point,
    
    \item the fixed vector $\alpha$ has height one, so $\hgt(\Lambda) = \ZZ$ and $\varphi[t]$ is a polynomial by Corollary~\ref{cor:p-eq-q-plus-1},

    \item the two evaluations of the equivariant $H^*$-polynomial are
    \[
    h^*_P(t) = 1 + t + t^3 + t^4, \qquad \varphi[t](\sigma) = 1 - t + 2t^2 - t^3 + t^4,
    \]
    so we have
    \[
    \varphi[t] = (1 + t^2 + t^4)\chi_0+ (t - t^2 + t^3)\chi_1.
    \]
\end{itemize}
In particular, the coefficient $\chi_0 - \chi_1$ of $t^2$ is not effective, so $\varphi[t]$ is a polynomial but is not effective. Hence Conjecture~\ref{conj:eff} does not hold. 
By the hierarchy in Theorem~\ref{thm:stapledon}, we have that the Ehrhart ring $S_P$ has no $\sigma$-stable l.s.o.p.\ of class $[\rho\otimes\CC]$, the polytope has no $\sigma$-invariant lattice triangulation, and no invariant nondegenerate polynomial.

\medskip\noindent\textbf{Consistency.}
As far as we are aware, the example is consistent with the literature and other conjectures. The second dilate $2P$ has $h^*_{2P} = 1+28t+78t^2+84t^3+57t^4+8t^5$ and $\varphi_{2P}[t](\sigma) = (1+t)^4$, which is effective, in accordance with \cite[Corollary~7.14]{Sta11}. Note that Proposition~\ref{prop:dilate} does not apply here, since the polytope $P$ is itself obstructing. The equivariant volume is $\varphi[1] = 3\chi_0 + \chi_1$, a permutation representation, in accordance with \cite[Conjecture~12.2]{Sta11}. But, as we note in Remark~\ref{rem:dead}, that agreement is not surprising. Finally, the eight vertices of $P$ do not affinely generate the affine lattice $M$, so \cite[Proposition~4.34]{Sta23} does not apply; see Corollary~\ref{cor:434-sharp}.
\end{example}

\begin{figure}[t]
\centering
\begin{tikzpicture}[scale=1.35, every node/.style={font=\small}]
  \coordinate (A) at (0,1.3);
  \coordinate (B) at (1.6,0);
  \coordinate (C) at (2.35,1.8);
  \coordinate (D) at (3.1,1.0);
  \coordinate (E) at (2.35,0.5);
  \fill[gray!10] (A) -- (B) -- (D) -- cycle;
  \fill[gray!20] (A) -- (C) -- (D) -- cycle;
  \draw[very thick] (A) -- (B) -- (D) -- cycle;
  \draw[very thick] (A) -- (C) -- (D);
  \draw[very thick,dashed] (B) -- (C);
  \foreach \p in {A,B,C,D} {\fill (\p) circle (0.06);}
  \fill (E) circle (0.06);
  \node[anchor=east]  at (A) {$(0,0,1)$};
  \node[anchor=north] at (B) {$(1,0,0)$};
  \node[anchor=south] at (C) {$(1,1,1)$};
  \node[anchor=west]  at (D) {$(1,2,0)$};
  \node[anchor=north west] at (E) {$(1,1,0)$};
  \node[anchor=west] at (4.3,0.9) {\begin{tabular}{l}The simplex $2P^\sigma$, in coordinates\\ $(s_1,s_2,s_3)$ with $s_j = \frac12(u_j+v_j)$.\\ Its normalised volume is $2$ and it has\\ five lattice points, with $h^*_{2P^\sigma} = 1+t$.\end{tabular}};
\end{tikzpicture}
\caption{Depiction of $2P^\sigma$ from Example~\ref{ex:counterexample}, where each vertex $2\pi(w)$ is labelled by its coordinates $(s_1,s_2,s_3)$. Each vertex of $P^\sigma$ is half-integral, and the fixed polytope $P^\sigma$ contains no lattice point.}
\label{fig:cex}
\end{figure}

The last example is another counterexample to the Effectiveness Conjecture. The effectiveness of the equivariant $H^*$-series fails in the same degree as
Example~\ref{ex:counterexample}, but for a different reason. In this example the polytope is not
$\sigma$-odd, and the obstruction lies on a proper face.

\begin{example}[Counterexample with $p \ge q+2$]\label{ex:counterexample-c2}
Let $n = 6$ and write coordinates for $\ZZ^7$ as $(x_1,x_2,u_1,v_1,u_2,v_2,h)$ with
$\widetilde M_0 = \{h = 0\}$ and vector $\alpha = (0,0,0,0,0,0,1)$. Let
$\sigma$ act by fixing the first two coordinates and exchanging the two pairs:
\[
\sigma(x_1,x_2,u_1,v_1,u_2,v_2,h) = (x_1,x_2,v_1,u_1,v_2,u_2,h).
\]
The type of $\sigma$ is $(3,0,2)$, so $p = 5$, $q = 2$ and $p = q+3$. Let $P \subseteq
\widetilde M$ be the convex hull of the eight points
\[
\begin{array}{llll}
w_1 = (-1,0,0,1,-1,0), & w'_1 = (-1,0,1,0,0,-1), &
w_2 = (0,0,0,1,1,1), & w'_2 = (0,0,1,0,1,1), \\[2pt]
w_3 = (1,-1,1,1,-1,-1), & w'_3 = (1,-1,1,1,0,0), \quad &
w_4 = (1,1,0,1,1,1), & w'_4 = (1,1,1,0,1,1),
\end{array}
\]
where each point is put at height one, i.e., setting the last coordinate $h = 1$. The involution $\sigma$ swaps $w_i$ and $w'_i$ for $i \in \{1,2,4\}$, while $w_3$ and $w'_3$ are each $\sigma$-fixed. So the eight points form three $\sigma$-orbits of size two together with two $\sigma$-fixed points. Then we have:
\begin{itemize}[]
    \item $P$ is $6$-dimensional with normalised volume
    $\Vol(P) = 4$, and its eight vertices are its only lattice points,
    \item the group $M_P$ has index two in $\ZZ^7$, so $P$ does not satisfy (Gen), and the fixed
    vector $\alpha$ has height one, so $\hgt(\Lambda) = \ZZ$,
    \item the polytope is not $\sigma$-odd because $M_P$ contains a $\sigma$-fixed
    element of odd height, but the four-dimensional face
    $
    F = P \cap \{u_1 + v_1 = 1\}
    $
    is $\sigma$-odd with
    $\dim F^\sigma = 2 = q$, so $F$ is obstructing,
    \item the two constituents of the fixed-point count are
    \[
    \Ev(k) = \tfrac1{48}k^4 + \tfrac16 k^3 + \tfrac{13}{24}k^2
             + \tfrac{13}{12}k + 1, \qquad
    \Od(k) = \Ev(k) - \tfrac14 k - \tfrac9{16},
    \]
    so $\deg \Ev = 4 = p-1$ but $\deg(\Ev-\Od) = 1 = q-1$, and $\varphi[t]$ is
    a polynomial by Proposition~\ref{prop:poly},
    \item the two evaluations of the equivariant $H^*$-series are
    \[
    h^*_P(t) = 1 + t + 2t^3 
    \quad \text{and} \quad 
    \varphi[t](\sigma) = 1 - t + 2t^2,
    \]
    so we have
    $
    \varphi[t] = (1 + t^2 + t^3)\chi_0 + (t - t^2 + t^3)\chi_1.
    $
\end{itemize}
In particular, the coefficient $\chi_0 - \chi_1$ of $t^2$ is not effective. So Conjecture~\ref{conj:eff} fails in case (C2).
\end{example}

\section{Searching for counterexamples}\label{sec:open}

In this section we explore what a \emph{minimal} counterexample to the Effectiveness Conjecture should look like, in the sense defined in Question~\ref{q:main} below, and give two reductions. The first reduction bounds the number of vertices for a minimal example, and the second allows us to pass to odd dilates. We then describe our search that produced Examples~\ref{ex:counterexample} and \ref{ex:counterexample-c2}.

\begin{proposition}[Necessary conditions for a counterexample]\label{prop:open}
Suppose $\varphi[t]$ is polynomial but not effective. Then the dimension satisfies $n \ge 3$, the polytope $P$ is not a simplex, and $P$ is not a unit cube with $\sigma$ permuting its coordinates. Moreover $P$ has no $\sigma$-invariant lattice triangulation, and if $P$ has $n+2$ vertices, then those vertices do not affinely generate the affine lattice $M$. The fixed polytope $P^\sigma$ is not a lattice polytope, the height group $\hgt(\Lambda)$ is all of $\ZZ$, and $p \ge q+1$. Finally $P$ has an obstructing face, and either:
\begin{enumerate}[label=(C\arabic*)]
    \item $p = q+1$, in which case the dimension $n \ge 4$ is even, $P$ is $\sigma$-odd and (Gen) does not hold, or
    \item $p \ge q+2$, in which case $\deg(\Ev-\Od) \le q-1$ while $\deg \Ev = p-1 \ge q+1$, so the coefficients of the difference $\Ev - \Od$ in degrees $q, \dots, p-2$ all vanish.
\end{enumerate}
\end{proposition}

\begin{proof}
Suppose that $\varphi[t]$ is polynomial but not effective. Then $P$ lies in none of the cases where Conjecture~\ref{conj:eff} is known to hold, which are listed in Section~\ref{sec:known}. So the dimension satisfies $n \ge 3$, the polytope $P$ is not a simplex, $P$ is not a unit cube with $\sigma$ permuting its coordinates, $P$ has no $\sigma$-invariant lattice triangulation, and if $P$ has $n+2$ vertices then those vertices do not affinely generate the lattice. By Corollary~\ref{cor:lattice-fixed}, the fixed polytope $P^\sigma$ is not a lattice polytope. By Corollary~\ref{cor:p-le-q}, we have $p \ge q+1$, and by Theorem~\ref{thm:lsop}, the polytope $P$ must have an obstructing face. Since $p > q$, if the height group were $2\ZZ$, then $\varphi[t]$ would not be polynomial by Corollary~\ref{cor:poly}(iv), so the height group is $\ZZ$. That proves every claim before the two cases.

It remains to establish the two cases, which we take according to the value of $p$. Consider first case (C1), so that $p = q+1$. By Corollary~\ref{cor:p-eq-q-plus-1}, the polytope $P$ is $\sigma$-odd and the condition (Gen) does not hold. In this case $n = 2q$ is even, and since $n \ge 3$, it follows that $n \ge 4$. So the claims of (C1) hold.

Consider now case (C2), so that $p \ge q+2$. By Proposition~\ref{prop:poly}, we have $\deg(\Ev-\Od)\le q-1$. By Lemma~\ref{lem:leading}, the polynomial $\Ev$ has degree $p-1 \ge q+1$. The bound $\deg(\Ev-\Od)\le q-1$ kills every coefficient of the difference $\Ev-\Od$ in degrees $q,\dots,p-2$. So the claims of (C2) hold.

We also record a bound that the statement does not claim. By Lemma~\ref{lem:fixed}(a), the number $c$ of copies of the matrix $U$ in the normal form of $\sigma$, see Section~\ref{sec:order-two}, satisfies $c \ge 1$. The search of Section~\ref{sec:search} uses that bound to restrict the possible types.
\end{proof}

In particular, a counterexample must have an obstructing face, so by Theorem~\ref{thm:lsop}, it must fail conditions (1), (2) and (3) of Stapledon's hierarchy (Theorem~\ref{thm:stapledon}). 
Case (C1) is realised by Example~\ref{ex:counterexample}, and case (C2) by Example~\ref{ex:counterexample-c2}.

Suppose that $\varphi[t]$ fails to be effective. Then some coefficient of $\varphi[t]$ is not effective. In the next remark, we narrow the range of degrees in which such a coefficient can lie in case (C1).

\begin{remark}[Degrees for which effectiveness can fail]\label{rem:mechanism}
Suppose $P$ lies in case (C1), so $p = q+1$ and $P$ is $\sigma$-odd. Then effectiveness can fail only in degrees $2, \dots, n$. For the upper end of that range, the polynomial $H$ of Lemma~\ref{lem:Hform} has degree at most $2p-1$ by \cite[Theorem~2.8]{Sta80}, so the evaluation $\varphi[t](\sigma) = (1+t)^{q-p}H(t)$ has degree at most $p+q-1 = n$, while $h^*_P$ has degree at most $n$ as well. Hence every coefficient of $\varphi[t]$ in degree greater than $n$ vanishes. Note that in degree zero the two evaluations are $(h^*_P)_0 = 1$ and $\varphi_0(\sigma) = 1$, so the constant coefficient is the trivial character $\varphi_0 = \chi_0$. For degree one, expanding Equation~\eqref{eq:evals} gives
\[
\varphi_1(\sigma) = N(1) + (q-p)N(0) = N(1) - 1.
\]
Here $N(1) = 0$ because a $\sigma$-fixed lattice point of $P$ would lie in $M_P \cap \Fix(\sigma)$ at height one, which contradicts the assumption that $P$ is $\sigma$-odd. So $\varphi_1(\sigma) = -1$. Also, since $P$ has no $\sigma$-fixed lattice points, every $\sigma$-orbit of a lattice point in $P$ has size two. Therefore, $P$ contains an even number of lattice points. Since $n+1 = 2q+1$ is odd, we have
\[
(h^*_P)_1 = \#(P \cap M) - (n+1)
\]
is odd and nonnegative. Hence $(h^*_P)_1 \ge 1$ and so 
\[
\varphi_1 = \frac{(h^*_P)_1 - 1}{2} \chi_0 + \left(1 + \frac{(h^*_P)_1 - 1}{2} \right)\chi_1
\]
is an effective character. We note that the polytope in Example~\ref{ex:counterexample} has as few vertices as possible with $\#(P\cap M) = n+2 = 8$ and $(h^*_P)_1 = 1$.
\end{remark}

\subsection{Reduction by monotonicity}\label{sec:monotonicity}
Observe that by Equation~\eqref{eq:evals}, the evaluation $\varphi[t](\sigma)$ depends only on $\sigma$ and $P^\sigma$. Using Stanley's monotonicity theorem \cite{Sta93} allows us to give a reduction to a polytope with fewer vertices. The same technique appears in the work of Bey, Henk and Wills on the bound \cite[Remark~1.6]{BHW}: a centrally symmetric lattice polytope contains a centrally symmetric lattice cross-polytope, and Stanley's monotonicity theorem then gives the bound. Stapledon uses a similar idea \cite[Example~2.16]{Sta23} by comparing the convex hull of a minimal set of orbits with the ambient polytope. The reduction from $P$ to a subpolytope $Q\subseteq P$ that we present keeps the same fixed polytope, hence the evaluations $\varphi_P[t](\sigma)$ and $\varphi_Q[t](\sigma)$ coincide. We note that this differs from the equivariant monotonicity in \cite[Theorem~1.7]{Sta23}, which compares two equivariant $H^*$-series directly but assumes the larger polytope has a $G$-invariant lattice triangulation that restricts to the smaller one. This strict hypothesis does not apply in either case (C1) or case (C2), while the reduction below uses only the classical monotonicity of $h^*$-polynomials.

\begin{proposition}[Monotonicity]\label{prop:monotone}
Let $Q \subseteq P$ be an $n$-dimensional $\sigma$-invariant lattice subpolytope with the same fixed polytope $Q^\sigma = P^\sigma$. Then $\varphi_Q[t](\sigma) = \varphi_P[t](\sigma)$. Moreover $\varphi_Q$ is polynomial if and only if $\varphi_P$ is polynomial, and $\varphi_P$ is effective if $\varphi_Q$ is effective.
\end{proposition}

\begin{proof}
By Equation~\eqref{eq:evals} the evaluation $\varphi[t](\sigma)$ depends only on $\sigma$ and on the fixed polytope $P^\sigma$. Since $Q$ is $n$-dimensional in the same ambient lattice and has $Q^\sigma = P^\sigma$, it follows that $\varphi_Q[t](\sigma) = \varphi_P[t](\sigma)$. The two evaluations at $\sigma$ agree, and the evaluation of an equivariant $H^*$-series at the identity is always a polynomial, so $\varphi_Q$ is polynomial if and only if $\varphi_P$ is polynomial. If $\varphi_Q$ is effective then, by Proposition~\ref{prop:eff} and Stanley's monotonicity theorem \cite{Sta93}, we have
\[
|\varphi_{P,i}(\sigma)| = |\varphi_{Q,i}(\sigma)| \le (h^*_Q)_i \le (h^*_P)_i \quad\text{for all } i,
\]
so $\varphi_P$ is effective.
\end{proof}

\begin{corollary}[Monotonic reduction]\label{cor:few-vertices}
Let $v$ be the number of vertices of $P^\sigma$. Then $P$ contains an $n$-dimensional $\sigma$-invariant lattice subpolytope $Q$ with the same fixed polytope $Q^\sigma = P^\sigma$ and at most $2(v+n)$ vertices. In particular, Conjecture~\ref{conj:eff} holds in general if it holds for each $\sigma$-invariant lattice polytope with at most $2(v+n)$ vertices, where $v$ is the number of vertices of its fixed polytope.
\end{corollary}

\begin{proof}
Let $v$ be the number of vertices of $P^\sigma$. For each vertex $w$ of $P^\sigma$, pick a vertex $v_w$ of $P$ such that $\pi(v_w) = w$, which exists by Lemma~\ref{lem:fixed}(a). Let $Q_0$ be the convex hull of all vertices $v_w$ and $\sigma(v_w)$. Then $Q_0$ is a $\sigma$-invariant subpolytope of $P$, and each vertex $w$ of $P^\sigma$ is the midpoint of the pair $v_w$ and $\sigma(v_w)$, so the point $w$ lies in $Q_0^\sigma$. Hence the fixed part $Q_0^\sigma \subseteq P^\sigma$ contains every vertex of $P^\sigma$, so the two fixed polytopes are equal. We form the polytope $Q$ as the convex hull of the vertices of $Q_0$ together with at most $n$ additional $\sigma$-orbits of lattice points of $P$, chosen so that $Q$ is $n$-dimensional; each added orbit raises the dimension of the convex hull by at least one, so at most $n$ orbits are needed. Adding those orbits does not change the fixed polytope. So $Q$ is an $n$-dimensional $\sigma$-invariant lattice subpolytope of $P$ with $Q^\sigma = P^\sigma$ and at most $2(v+n)$ vertices, which is the first claim.

By Proposition~\ref{prop:monotone}, the series $\varphi_P[t]$ is effective whenever $\varphi_Q[t]$ is effective, and $\varphi_P[t]$ is polynomial exactly when $\varphi_Q[t]$ is polynomial. Hence Conjecture~\ref{conj:eff} holds in general if the conjecture holds for every $\sigma$-invariant lattice polytope with at most $2(v+n)$ vertices, where $v$ is the number of vertices of its fixed polytope.
\end{proof}

The lowest-dimensional instance of case (C2) has $n = 3$ and type $(2,0,1)$. To check the Effectiveness Conjecture there, it suffices to search for polytopes whose vertices consist of one $\sigma$-orbit over each vertex of the fixed polygon, together with at most three extra orbits so that the polytope has the correct dimension. The search described in Section~\ref{sec:search} does not cover that instance, which we raise as Question~\ref{q:dim3}.

\subsection{Reduction by dilation}\label{sec:dilates} 
By \cite[Remark~7.2]{Sta11}, the equivariant $H^*$-series of a dilate $mP$ is a polynomial whenever the exponent of the group, the least common multiple of the orders of its elements, divides $m$. In our case, the group has order two so every even dilate of $P$ has polynomial equivariant $H^*$-series. By \cite[Corollary~7.14]{Sta11}, the $H^*$-series is effective if the order of the group divides $m$. The proof of that corollary uses \cite[Corollary~7.10]{Sta11}, so the erratum in \cite{Sta23} lists it among the results whose published proofs are incomplete. However, this result can be deduced by combining \cite[Theorem~1.4 and Theorem~1.5]{Sta23}. In the proposition below, we show that the property of having no obstructing faces is preserved by taking odd dilations.

\begin{proposition}[Dilates]\label{prop:dilate}
Let $m$ be an odd positive integer.
\begin{enumerate}[label=(\alph*)]
    \item The series $\varphi_{mP}$ is polynomial if and only if $\varphi_P$ is polynomial.
    \item If $P$ has no obstructing face, then the dilate $mP$ has no obstructing face either.
\end{enumerate}
Moreover, if $P$ has no obstructing face, then $\varphi_{kP}$ is effective for every positive integer $k$.
\end{proposition}

\begin{proof}
(a) Let $m$ be an odd positive integer. The involution $\sigma$ does not change when passing to a dilate, so the numbers $p$ and $q$ are the same for $P$ and for the dilate $mP$, and the fixed lattice point count satisfies $N_{mP}(k) = \#(km P^\sigma\cap\ZZ^{n+1}) = N_P(mk)$. Since $m$ is odd, the parity of $mk$ is the same as that of $k$, so $k\mapsto \Ev(mk)$ and $k\mapsto \Od(mk)$ are the constituents of the fixed lattice point count of $mP$. Their difference is $(\Ev-\Od)(mk)$, which has the same degree as the difference for $P$. Since the criterion of Proposition~\ref{prop:poly} involves only that degree and the number $q$, which is unchanged, the series $\varphi_{mP}$ is polynomial if and only if $\varphi_P$ is polynomial.

(b) Suppose that $P$ has no obstructing face. Each face of $mP$ is of the form $mF$ for a face $F$ of $P$ and has $(mF)^\sigma = mF^\sigma$. In particular $(mF)^\sigma$ and $F^\sigma$ have the same dimension. Suppose the face $F$ is not $\sigma$-odd. If $\sigma(F) \neq F$, then $\sigma(mF) \neq mF$, so $mF$ is not $\sigma$-odd. Otherwise $M_F$ contains a $\sigma$-fixed element $u$ of odd height, so the element $mu$ lies in the group $M_{mF}$, is $\sigma$-fixed, and has odd height because $m$ is odd. Hence $mF$ is not $\sigma$-odd in either case. So every $\sigma$-odd face of $mP$ comes from a $\sigma$-odd face of $P$. Since the dilation respects the dimension of the fixed part of a face, it follows that $mP$ has no obstructing face.

For the last claim, let $k$ be a positive integer. If $k$ is odd, then by part (b) the dilate $kP$ has no obstructing face, so $\varphi_{kP}[t]$ is effective by Theorem~\ref{thm:lsop}. If $k$ is even, then effectiveness of $\varphi_{kP}[t]$ follows from \cite[Corollary~7.14]{Sta11}.
\end{proof}

In particular, by \cite[Corollary~7.14]{Sta11}, no counterexample to Conjecture~\ref{conj:eff} has the form $2P$ for a lattice polytope $P$. For instance, in Example~\ref{ex:counterexample} the second dilate has $\varphi_{2P}[t](\sigma) = (1+t)^4$, while the coefficients of $h^*_{2P}$ are much larger: the $h^*$-vector of $2P$ is $(1, 28, 78, 84, 57, 8)$. Hence $\varphi_{2P}[t]$ is effective by Proposition~\ref{prop:eff}.

\subsection{Randomised search}\label{sec:search}

In this section, we describe how Examples~\ref{ex:counterexample} and \ref{ex:counterexample-c2} were found.
Proposition~\ref{prop:open} tells us that counterexamples fall into one of two types (C1) and (C2), and Corollary~\ref{cor:few-vertices} bounds the number of vertices of minimal counterexamples.
Our general approach is to pick a candidate fixed polytope $R$ and compute $\varphi[t](\sigma)$ from $R$. Then we search for small lattice polytopes $P$ with $P^\sigma = R$. Recall from Section~\ref{sec:order-two} the normal form of $\sigma$ of type $(a,b,c)$, together with the basis of $\ZZ^{n+1}$ given by: $x_1,\dots,x_a$ that are pointwise fixed by $\sigma$; $y_1,\dots,y_b$ that are each negated by $\sigma$; and
$z_1, \dots, z_c, w_1, \dots, w_c$ where $\sigma$ swaps $z_i$ and $w_i$. The fixed space $\Fix(\sigma)$ then has coordinates $(x,s)$, where $x = (x_1,\dots,x_a)$ and $s_j = \frac12(z_j+w_j)$ for $1 \le j \le c$. In these coordinates the fixed polytope $R$ has vertices in $\ZZ^a \times (\frac12\ZZ)^c$, so we work with the dilate $2R$, whose vertices lie in $(2\ZZ)^a \times \ZZ^c$. Recall also that $R$ lies in the affine hyperplane $\{\hgt = 1\}$ of $\Fix(\sigma)$ and has dimension $p-1 = a+c-1$, so the dilate $2R$ lies in $\{\hgt = 2\}$.

In the first stage, we construct a candidate $2R$ as the convex hull of $k$ points of height two in $(2\ZZ)^a \times \ZZ^c$, with $a+c+1 \le k \le a+c+3$. Spanning the affine hull of $2R$ takes at least $a+c$ points, so the lower limit means we always take more points than a simplex of that dimension has vertices. We take these $k$ points uniformly at random from a large box. We compute the Ehrhart series of $R$, which gives us the evaluation of the equivariant $H^*$-series of the final polytope $\varphi[t](\sigma) = (1-t)^p(1+t)^q\,\mathrm{Ehr}_R(t)$. We immediately discard the candidate if any of the following occur: 
\begin{itemize}
    \item $R$ is a lattice polytope, see Corollary~\ref{cor:lattice-fixed},
    \item $\varphi[t](\sigma)$ is not a polynomial,
    \item in the case (C1), if $R$ contains a lattice point in $\ZZ^{a+c}$, then $P$ has a $\sigma$-fixed lattice point and is not $\sigma$-odd,
    \item in case (C2), if $R$ has no face $\Gamma$ with $\dim \Gamma \ge q$ and $\Gamma \cap \ZZ^{a+c} = \emptyset$.
\end{itemize}
We note that the last condition is necessary because an obstructing face $F$ of $P$ is $\sigma$-invariant and $\sigma$-odd, so $F^\sigma$ is a face of $R$ of dimension at least $q$ that has no lattice point.
For case (C2), we modify the search so that at least $a+c$ of the chosen points of $2R$ lie on the hyperplane $\{s_1 = 1\}$ and the remaining points lie in $\{s_1 \ge 2\}$.
The corresponding facet of $R$ then lies in $\{s_1 = \frac12\}$, so that facet
contains no lattice point, hence $R$ has a facet free of lattice points.

The next stage is to lift $R$ to a minimal lattice polytope. By Corollary~\ref{cor:few-vertices}, it suffices to find one $\sigma$-orbit for each vertex of $R$, together with at most $n$ further orbits.
Let $(\xi, s)$ be a vertex of $R$ in the coordinates above. The fibre of the projection map $\pi$ from Section~\ref{sec:order-two} over this point is parametrised by $(y,d) \in \ZZ^b \times \ZZ^c$ with $d_j \equiv 2s_j \pmod 2$. Explicitly, the points are
\[
(x,y,z,w) = \left(\xi, y, s + \tfrac12 d, s - \tfrac12 d\right).
\]
We consider all points in the fibre whose entries are bounded in absolute value by some constant $D$. If there are relatively few options then we go through them exhaustively, otherwise we sample them 
uniformly. 
So, by construction, the convex hull of these orbits has fixed polytope $R$ and dimension
$\dim R + \operatorname{rank}(\Phi) = p - 1 + \operatorname{rank}(\Phi)$, where $\Phi$ is the matrix whose columns are the fibre vectors $(y,d) \in \ZZ^{b+c}$ for each orbit. So the lift
is $n$-dimensional if the chosen fibre vectors span $\RR^{b+c}$. If they do not span, then we add further orbits over vertices of $R$ until the polytope is $n$-dimensional, which does not change the fixed polytope.
Note that if the vertices of a lift are all congruent modulo two, then $v + \sigma(v) \equiv 2v \equiv 0$ modulo two for each vertex $v$, so every vertex $\pi(v)$ of the fixed polytope is a lattice point. The series of such a lift is then effective by Corollary~\ref{cor:lattice-fixed}, and we discard it. In case (C2), we discard the lift if the polytope $P$ has no obstructing
face, since otherwise $\varphi[t]$ is effective by Theorem~\ref{thm:lsop}.
Finally, we compute $h^*_P$ and compare its coefficients with $|\varphi_i(\sigma)|$ to determine effectiveness.

We computed the Ehrhart series using \textit{Normaliz~3.11.0}
\cite{Normaliz} through \textit{PyNormaliz~2.24}.  
The number of
candidates and lifts examined in each case is recorded in Table~\ref{tab:search}. We conclude this section with the questions that follow from the computations.

\begin{table}[t!]
\centering
\setlength{\tabcolsep}{4pt}
\begin{tabular}{@{}ccc rrrr@{}}
\toprule
\shortstack{Type\\$(a,b,c)$} & $n$ & $(p,q)$ & \shortstack{Fixed\\polytopes} & Lifts
  & \shortstack{Non-effective\\lifts} & \shortstack{Fixed polytopes\\with a non-\\effective lift} \\
\midrule
\multicolumn{7}{@{}l}{\emph{Case} (C1)} \\
$(1,0,2)$ & $4$ & $(3,2)$ & 88 & 2\,025 & 0 & 0 \\
$(2,1,1)$ & $4$ & $(3,2)$ & 0 & 0 & 0 & 0 \\
$(1,0,3)$ & $6$ & $(4,3)$ & 11\,746 & 55\,543 & 1\,416 & 480 \\
$(2,1,2)$ & $6$ & $(4,3)$ & 659 & 175\,290 & 12\,012 & 360 \\
$(3,2,1)$ & $6$ & $(4,3)$ & 139 & 139\,000 & 0 & 0 \\
$(1,0,4)$ & $8$ & $(5,4)$ & 4\,546 & 220\,917 & 8\,469 & 554 \\
\midrule
\multicolumn{7}{@{}l}{\emph{Case} (C2)} \\
$(2,0,1)$ & $3$ & $(3,1)$ & 0 & 0 & 0 & 0 \\
$(3,0,1)$ & $4$ & $(4,1)$ & 4\,146 & 30\,013 & 0 & 0 \\
$(2,0,2)$ & $5$ & $(4,2)$ & 927 & 1\,596 & 0 & 0 \\
$(4,0,1)$ & $5$ & $(5,1)$ & 25 & 25 & 0 & 0 \\
$(3,1,1)$ & $5$ & $(4,2)$ & 0 & 0 & 0 & 0 \\
$(3,0,2)$ & $6$ & $(5,2)$ & 454 & 1\,437 & 51 & 26 \\
$(4,1,1)$ & $6$ & $(5,2)$ & 273 & 57\,765 & 247 & 43 \\
$(5,0,1)$ & $6$ & $(6,1)$ & 1 & 1 & 0 & 0 \\
$(2,0,3)$ & $7$ & $(5,3)$ & 725 & 7\,130 & 140 & 26 \\
$(4,0,2)$ & $7$ & $(6,2)$ & 4 & 20 & 0 & 0 \\
$(3,0,3)$ & $8$ & $(6,3)$ & 75 & 2\,136 & 16 & 4 \\
$(2,0,4)$ & $9$ & $(6,4)$ & 173 & 10\,320 & 138 & 8 \\
\bottomrule
\end{tabular}
\caption{Candidates from the search in Section~\ref{sec:search}. 
The \emph{Fixed polytopes} column counts the
candidates $R$ that passed the first stage and \emph{Lifts} records the lifts for which $h^*_P$ was computed. 
For types $(2,1,1)$ in case (C1), and for type $(3,1,1)$ in case (C2), no candidate passed the first stage.}
\label{tab:search}
\end{table}

\begin{question}[Minimal counterexamples]\label{q:main}
Both Examples~\ref{ex:counterexample} and \ref{ex:counterexample-c2} have $n+2$ lattice points, normalised volume four, and an
$h^*$-polynomial with coefficient zero in degree two. Call a counterexample \emph{minimal} if it has no proper $n$-dimensional $\sigma$-invariant lattice subpolytope with the same fixed polytope, so that Proposition~\ref{prop:monotone} gives no further reduction. Is a minimal counterexample forced to have these properties?
\end{question}

\begin{question}[Dimension three]\label{q:dim3}
For $n = 3$ and type $(2,0,1)$, is there a lattice polytope where $\varphi[t]$ is polynomial but $P$ has an obstructing face? 
\end{question}

\begin{question}[Least dimension]\label{q:small}
What is the least dimension of a counterexample to Conjecture~\ref{conj:eff}? By Proposition~\ref{prop:open} it is at least three, and by Proposition~\ref{prop:open}(C1) a counterexample in case (C1) has even dimension at least four. Examples~\ref{ex:counterexample} and \ref{ex:counterexample-c2} give an upper bound of six.
\end{question}

\begin{question}[Classification]\label{q:classify}
Is it possible to classify $\sigma$-invariant lattice polytopes in cases (C1) and (C2) that fail to be effective? 
\end{question}

\section{Proof of Theorem~\ref{thm:lsop}}\label{sec:proofs}

In this section we prove Theorem~\ref{thm:lsop}.
Before the proof, we give some preliminaries on toric varieties.
Throughout, the group $G = \langle\sigma\rangle$ has order two and we use the notation of Section~\ref{setup} and Section~\ref{sec:order-two}: the map $\hgt$ is the height, the lattice $\Lambda = (\ZZ^{n+1})^\sigma$ is the fixed lattice, the numbers $p$ and $q$ are the dimensions of the $(+1)$ and $(-1)$-eigenspaces of the involution $\sigma$ respectively, the map $\pi = \frac12(I+\sigma)$ is the projection onto the fixed space, and the ring $S_P$ is the Ehrhart ring.

\subsection{Preliminaries on toric varieties}\label{sec:toric-prelim}

Let us fix the notation for this section. Let $X = \Proj S_P$ be the projective toric variety of the $n$-dimensional polytope $P$ and for a subset $T \subseteq (S_P)_1$ write
\[
V(T) := \{\xi \in X \colon s(\xi) = 0 \text{ for all } s \in T\}
\]
for its vanishing locus. We write $V(s_1,\dots,s_r)$ if $T$ is the $\CC$-linear span of $s_1, \dots, s_r$. When $T$ is a linear subspace we call $V(T)$ the \emph{base locus} of $T$. For a nonempty face $F$ of $P$, let $C_F = \RR_{\ge0}F$ be the cone over $F$ and
$
L_F = \Span_\RR(F)\cap\ZZ^{n+1}
$ be the lattice contained in the linear span of $F$.

The variety $X$ comes with an action of the torus $\TT = \Hom(\ZZ^{n+1}\cap\{\hgt = 0\},\CC^*)$ with one orbit for each nonempty face $F$ of $P$. We write $\orb{F}$ for the orbit associated to $F$. Recall that $\dim \orb{F} = \dim F$ and $X$ is the disjoint union $\bigsqcup_F \orb{F}$ over the nonempty faces of $P$. That decomposition is the well-known orbit-face correspondence \cite[Theorem~3.2.6]{CLS}, and Lemma~\ref{lem:toric}(a) below is a specific form that we will use. We also use Krull's principal ideal theorem and Krull's height theorem in the graded setting, for which see \cite[Section~1.2 and Appendix~A]{BH}.

We record two facts about the lattice $L_F$ and the group $M_F$, used in the proof of Lemma~\ref{lem:toric}(a) and in the proof of Lemma~\ref{lem:Z}(b). Firstly, the lattice $L_F$ is generated by the lattice points of the cone $C_F$. To see this, take any lattice point $v \in \Span_\RR(F)$ and a lattice point $a$ in the relative interior of $C_F$. Then $v + ka$ lies in $C_F$ for all $k$ large enough, so $(v+ka) - ka$ lies in the group generated by the lattice points of $C_F$. Secondly, the group $M_F$, which is the group generated by the lattice points of $F$ (Definition~\ref{def:odd}), has finite index in $L_F$, because the vertices of $F$ span $\Span_\RR F$. It is helpful to think of $L_F$ as the \emph{saturation} of $M_F$: the lattice $L_F$ consists of every lattice point of $\Span_\RR F$, so passing from $M_F$ to $L_F$ fills in the lattice points of that span which $M_F$ misses.

\begin{lemma}[Points and orbits of $X$]\label{lem:toric}
We have the following. 
\begin{enumerate}[label=(\alph*)]
    \item The closed points of $X$ correspond to the pairs $(F,[\chi])$ where $F$ is a nonempty face of $P$ and the character $\chi \in \Hom(L_F,\CC^*)$ is taken modulo the relation $\chi \sim \lambda^{\hgt}\chi$ for $\lambda\in\CC^*$. Here $\lambda^{\hgt}$ denotes the character $u \mapsto \lambda^{\hgt(u)}$ of $L_F$. The pairs with a fixed face $F$ give the $\dim F$-dimensional orbit $\orb{F}$. For a lattice point $u \in P \cap M$ and a point $\xi = (F,[\chi])$, the evaluation of $x^u \in (S_P)_1$ at $\xi$ is $x^u(\xi) = \chi(u)$ if $u$ lies in $F$, and vanishes otherwise. These values are well defined up to a common scalar. In particular $x^v(\xi)$ is nonzero for each vertex $v$ of $F$.
    \item For elements $s_1,\dots,s_r \in (S_P)_1$, the quotient $S_P/(s_1,\dots,s_r)$ is finite-dimensional if and only if $s_1, \dots, s_r$ have no common zero on $X$.
    \item Let $Y \subseteq X$ be irreducible and closed and let $s \in (S_P)_1$. If the element $s$ vanishes on $Y$, then the intersection is $Y \cap V(s) = Y$. Otherwise, every component of $Y \cap V(s)$ has dimension exactly $\dim Y - 1$. So, if the element $s$ does not vanish on $Y$, then the intersection $Y \cap V(s)$ is empty when $\dim Y = 0$, and is nonempty of dimension $\dim Y - 1$ when $\dim Y \ge 1$.
\end{enumerate}
\end{lemma}

\begin{proof}
(a) We describe the closed points of the affine cone over $X$ and then pass to the quotient by the grading action. The closed points of the affine cone $\operatorname{Spec}S_P$ correspond to semigroup homomorphisms $\psi\colon C_P\cap\ZZ^{n+1}\to(\CC,\cdot)$ with $\psi(0) = 1$. Explicitly, the closed point is the maximal ideal obtained from the kernel of the ring map $\tilde \psi : S_P \rightarrow \CC$ where $x^u \mapsto \psi(u)$. The locus where $\psi$ is nonzero is a \emph{face} of the semigroup, that
is, a subsemigroup containing both summands of each of its elements. Equivalently, the set of lattice points on which $\psi$ vanishes is a prime ideal of the semigroup. Since the semigroup is saturated, the face is of the form $C_F\cap\ZZ^{n+1}$ for a face $F$ of $P$. On $F$, the homomorphism extends uniquely to a character of $L_F$. 
Points of $X$ correspond to such $\psi$ with nonempty face $F$ modulo the action of $\CC^*$ induced by the grading, namely $\psi \mapsto \lambda^{\hgt}\psi$,
which is the relation in the statement. See \cite[Chapters~1--3]{CLS}.
The pairs whose first entry is the face $F$ are exactly the points of the orbit $\orb F$, of dimension $\dim F$, by the orbit-face correspondence \cite[Theorem~3.2.6]{CLS}.
Under that correspondence the map $\tilde\psi$ sends $x^u$ to $\chi(u)$ for a lattice point $u$ of $F$, and sends $x^u$ to zero for a lattice point $u$ of $P$ outside $F$, which is the stated evaluation of $x^u$ at $\xi$. Each vertex $v$ of $F$ is a lattice point of $F$ and the character $\chi$ takes values in $\CC^*$, so the value $x^v(\xi) = \chi(v)$ is nonzero.

(b) This follows directly from the graded Nullstellensatz. See \cite[Theorem~1.5.5]{BH}. 

(c) If $\dim Y = 0$, then $Y$ is a single closed point of $X$ and the intersection $Y \cap V(s)$ is $Y$ when $s$ vanishes at that point and is empty otherwise, which is the claim. So assume $\dim Y \ge 1$. The set $Y$ corresponds to a homogeneous prime $\mathfrak p$ with $\dim S_P/\mathfrak p = \dim Y + 1 \ge 2$. By Krull's principal ideal theorem, every minimal prime $\mathfrak q$ over the ideal $\mathfrak p + (s)$ satisfies $\dim S_P/\mathfrak q \ge \dim Y$, with equality when the element $s$ lies outside $\mathfrak p$, because the quotient $S_P/\mathfrak p$ is an affine domain. Such prime ideals are homogeneous and do not contain the irrelevant ideal, so they correspond to nonempty components of the intersection $Y\cap V(s)$. If the element $s$ lies outside $\mathfrak p$, then that intersection is a proper closed subset of $Y$. Hence, if the element $s$ does not vanish on $Y$, then every component of the intersection $Y\cap V(s)$ has dimension exactly $\dim Y - 1$, and the intersection is nonempty because $\dim Y \ge 1$.
\end{proof}

\begin{proposition}[{\cite[Lemma~4.8]{Sta23}}, {\cite[Theorem~1.4]{Ste94}}]\label{prop:lsop-eff}
Let $G$ be a finite group and let $V \subseteq (S_P)_1$ be a $G$-stable l.s.o.p.\ of class $[V] = [\rho\otimes\CC]$. Then the equivariant $H^*$-series is $\varphi[t] = \sum_k[(S_P/(V))_k]t^k$, in particular $\varphi[t]$ is effective.
\end{proposition}

\begin{proof}
The cited lemma applies to a Cohen--Macaulay ring, and the ring $S_P$ is Cohen--Macaulay \cite{Hoc72}. A basis of the l.s.o.p.\ $V$ is a homogeneous system of parameters of degree one, so \cite[Lemma~4.8]{Sta23} applied to that basis gives the result.
\end{proof}

\subsection{Main proof}

Let $W_+$ and $W_-$ be the $(\pm1)$-eigenspaces of $\sigma$ on the degree-one part $(S_P)_1$, so that $W_+$ is spanned by the elements $x^u + x^{\sigma u}$ and $W_-$ by the elements $x^u - x^{\sigma u}$, for $u$ a lattice point of $P$. Let
\[
Z_+ := V(W_-) \subseteq X, \qquad Z_- := V(W_+) \subseteq X,
\]
so $Z_+$ is the base locus of $W_-$ and $Z_-$ the base locus of the invariant system $W_+$. We use this notation so that $Z_+$ are the points where the forms $f$ that are \emph{anti-invariant}, i.e., $\sigma$ maps $f$ to $-f$, vanish. Each locus is $\sigma$-invariant and closed in $X$.

\begin{lemma}[Base loci]\label{lem:Z}
We have the following.
\begin{enumerate}[label=(\alph*)]
    \item The two loci are disjoint, $Z_+\cap Z_- = \emptyset$.
    
    \item The intersection $Z_+\cap \orb{F}$ is nonempty if and only if $F$ satisfies $\sigma(F) = F$, and in that case, the dimension of $Z_+\cap \orb{F}$ is $\dim F^\sigma$. Hence $Z_+$ has dimension $p-1$.
    
    \item The intersection $Z_-\cap \orb{F}$ is nonempty if and only if $F$ is $\sigma$-odd, and in that case the dimension of $Z_-\cap \orb{F}$ is $\dim F^\sigma$. Hence the dimension of $Z_-$ is the maximum of $\dim F^\sigma$ over the $\sigma$-odd faces $F$, and is $-\infty$ if there are none.
\end{enumerate}
\end{lemma}

\begin{proof}
Let $\xi = (F,[\chi])$ be a point of the orbit $\orb{F}$. Throughout the proof we use Lemma~\ref{lem:toric}(a).

(a) Assume that $\xi$ lies in both loci. Then $x^u(\xi) = x^{\sigma u}(\xi) = -x^{\sigma u}(\xi)$, so the value $x^u(\xi)$ vanishes for each lattice point $u$ of $P$. But $x^v(\xi)$ is nonzero for each vertex $v$ of $F$, a contradiction.

(b) The point $\xi$ lies in $Z_+$ if and only if $x^u(\xi) = x^{\sigma u}(\xi)$ for all lattice points $u$ of $P$. If $F$ is not invariant, then there exists a vertex $u$ of $F$ such that $\sigma u$ lies outside $F$. So we have $x^u(\xi)\ne0 = x^{\sigma u}(\xi)$, hence the intersection $Z_+ \cap \orb{F}$ is empty. 
On the other hand, if $\sigma(F) = F$, then for any lattice point $u$ of $P$ outside of $F$, its image $\sigma u$ is also outside $F$, so $x^u(\xi) = 0 = x^{\sigma u}(\xi)$. Hence the condition defining $Z_+$ reduces to $\chi(\sigma u) = \chi(u)$ for the lattice points $u$ of $F$, that is, the character $\chi$ is trivial on $(\sigma-1)M_F = \{\sigma u - u : u \in M_F\} \subseteq L_F$. The characters satisfying this are given by 
\[
K := \Hom\big(L_F/(\sigma-1)M_F,\ \CC^*\big),
\]
of dimension $\rank L_F - \rank(\sigma-1)M_F$. Since $M_F$ has finite index in $L_F$, the rank of $(\sigma-1)M_F$ is the dimension of the $(-1)$-eigenspace of $\sigma$ acting on $\Span_\RR F$. So $K$ has dimension $\dim(\Span_\RR F)^\sigma = \dim F^\sigma + 1$ by Lemma~\ref{lem:fixed}(b). 

The group $K$ is not yet the intersection $Z_+\cap\orb{F}$, because characters are taken modulo the scaling action. We now check that the scaling action preserves $K$ and is free. Since height is $\sigma$-invariant, each element $\sigma u - u$ of
$(\sigma-1)M_F$ has height zero. Hence the character $\lambda^{\hgt}$ is trivial on $(\sigma-1)M_F$, and the scaling action $\chi\mapsto\lambda^{\hgt}\chi$ maps $K$ to
itself. The action is also free: if $\lambda^{\hgt}\chi = \chi$, then $\lambda^{\hgt(u)} = 1$ for every element $u$ of $L_F$, and the height map is surjective onto
$\ZZ$ because the vertices of $F$ have height one, thus $\lambda = 1$. The intersection
$Z_+\cap\orb{F}$ is therefore the quotient of $K$ by the scaling action. The group $K$ contains the trivial character, so the intersection $Z_+\cap\orb{F}$ is nonempty, and the quotient has dimension
$\dim K - 1 = \dim F^\sigma$.

For the last claim, $X$ is the union of its finitely many orbits, so the
dimension of $Z_+$ is the maximum of $\dim F^\sigma$ over the
$\sigma$-invariant faces $F$. By Lemma~\ref{lem:fixed}(c), the whole polytope has $\dim P^\sigma = p-1$ and every
proper invariant face $F$ has $\dim F^\sigma \le p-2$. Hence $\dim Z_+ = p-1$.

(c) The point $\xi$ lies in $Z_-$ if and only if $x^u(\xi) = -x^{\sigma u}(\xi)$ for all lattice points $u$ of the polytope, which, by the same argument as part (b), forces the face to be invariant. 
For a lattice point $u$ of $F$, this means $\chi(u) = -\chi(\sigma u)$, that
is,
$
\chi(\sigma u - u) = -1
$.
We record the condition $\chi(\sigma u - u) = -1$ as a single equation between two characters of the group $M_F$. Let
\[
\Psi\colon\Hom(L_F,\CC^*)\to\Hom(M_F,\CC^*), \qquad \text{defined by }
\Psi(\chi) = \chi\circ(\sigma-1)|_{M_F} ,
\]
and let $\varepsilon \in \Hom(M_F,\CC^*)$ be the character $\varepsilon(m) = (-1)^{\hgt(m)}$.

The lattice points of $F$ have height one, so
$\varepsilon$ takes the value $-1$ on each of them. The lattice points of $F$ generate $M_F$, and both $\Psi(\chi)$ and $\varepsilon$ are homomorphisms, so $\chi(u) = -\chi(\sigma u)$
holds for every lattice point $u$ of $F$ exactly when $\Psi(\chi) = \varepsilon$. So the characters $\chi$ for which the point $(F,[\chi])$ lies in $Z_-$ are exactly the solutions of $\Psi(\chi) = \varepsilon$. Next we determine when a solution exists, and then compute the dimension of the set of solutions.

A solution exists if and only if the character $\varepsilon$ lies in the image of $\Psi$, so we show that the image of $\Psi$ is the set of characters of $M_F$ acting trivially on $M_F\cap\Fix(\sigma)$.
Firstly, every character in the image of $\Psi$ is of this kind, because each element $m$ of $M_F\cap\Fix(\sigma)$ satisfies $(\sigma-1)m = 0$, so $\Psi(\chi)(m) = \chi(0) = 1$ for every $\chi$. Conversely, the subgroup $M_F\cap\Fix(\sigma)$ is the kernel of $(\sigma-1)|_{M_F}$, so a character $\psi$ of $M_F$ that is trivial on that kernel factors as $\psi = \psi'\circ(\sigma-1)$ for a character $\psi'$ of the subgroup
$(\sigma-1)M_F$. The character $\psi'$ extends to the lattice $L_F$ because the group $\CC^*$ is \emph{divisible}, i.e., every element of $\CC^*$ has a $k$-th root for every positive integer $k$, and any such extension is a preimage of $\psi$.
The character $\varepsilon$ is trivial on the subgroup $M_F\cap\Fix(\sigma)$ exactly when every $\sigma$-fixed element of $M_F$ has even height, which is the definition of a $\sigma$-odd face. So a solution exists if and only if the face $F$ is $\sigma$-odd, which is the first claim of part (c).

Suppose now that the face $F$ is $\sigma$-odd, so that a solution exists, and let us compute the dimension of the set of solutions. The kernel of $\Psi$ coincides with the group $K$ of part
(b), so the set of solutions of $\Psi(\chi) = \varepsilon$ is a
coset of $K$, of dimension $\dim K = \dim F^\sigma + 1$. The character $\lambda^{\hgt}$ lies in $K$, because every element $\sigma u - u$ of $(\sigma-1)M_F$ has height zero, and multiplication by an element of $K$ sends each coset of $K$ to itself. So the scaling action $\chi\mapsto\lambda^{\hgt}\chi$ preserves the coset of solutions, and that action is free by the argument in part (b). Hence the
intersection $Z_-\cap\orb{F}$ has dimension $\dim K - 1 = \dim F^\sigma$.

For the last claim, $X$ is the union of its finitely many orbits, so the dimension of $Z_-$ is the maximum of $\dim F^\sigma$ over the $\sigma$-odd faces $F$. If no face of $P$ is $\sigma$-odd, then every intersection $Z_-\cap\orb{F}$ is empty, so the locus $Z_-$ is empty and its dimension is $-\infty$.
\end{proof}

We will now give a proof of our main result.

\begin{proof}[Proof of Theorem~\ref{thm:lsop}]
We prove both directions by studying the two base loci as in Lemma~\ref{lem:Z}. For the forward direction we build the subspace $V$ in two parts. In the first part we choose $q$ anti-invariant elements $\theta_1,\dots,\theta_q$ of $W_-$, cutting $X$ down to a locus $Z$ of dimension at most $p-1$ that is disjoint from $Z_-$. In the second part we choose $p$ invariant elements $\phi_1,\dots,\phi_p$ of $W_+$, cutting $Z$ down to the empty set. We use the assumption that $P$ has no obstructing face in the first part for the bound $\dim Z_- \le q-1$. For the converse direction we take an obstructing face and produce a common zero of every candidate subspace.

\medskip
\noindent $(\Rightarrow)$
Any $\sigma$-stable subspace $V \subseteq (S_P)_1$ splits as $V = V_+\oplus V_-$ with $V_+ \subseteq W_+$ and $V_- \subseteq W_-$ and its class in the representation ring is $[V] = (\dim V_+)\chi_0 + (\dim V_-)\chi_1$. Recall from Definition~\ref{def:odd} that a $\sigma$-odd face $F$ is obstructing if its fixed part satisfies $\dim F^\sigma \ge q$. So, if $P$ has no obstructing faces, then by Lemma~\ref{lem:Z}(c), $Z_-$ has dimension at most $q-1$.

Assume that $P$ has no obstructing face. We choose elements $\theta_1,\dots,\theta_q\in W_-$ inductively, writing $Z^{(j)} = X \cap V(\theta_1,\dots,\theta_j)$ for each $j$, satisfying:
\begin{itemize}[nosep]
    \item every component of $Z^{(j)}$ not contained in $Z_+$ has dimension at most $n-j$, and
    \item every component of the intersection $Z_-\cap Z^{(j)}$ has dimension at most $\dim Z_- - j$.
\end{itemize}
The role of the first condition is to bound the dimension of the locus from the first part. The second condition forces the intersection $Z_-\cap Z^{(j)}$ to be empty for $j=q$, so that we can move on to the second part.

For $j = 0$, we have $Z^{(0)} = X$ is irreducible of dimension $n$ and both conditions hold.
Suppose the conditions hold for $j-1$, and let $Y_1,\dots,Y_r$ be the components of $Z^{(j-1)}$ not contained in $Z_+$, together with the components of the intersection $Z_-\cap Z^{(j-1)}$. Before choosing $\theta_j$, we record what we need about those components.
Firstly, no component of $Z_-\cap Z^{(j-1)}$ is contained in $Z_+$, by Lemma~\ref{lem:Z}(a).
Secondly, no component of $Z^{(j-1)}$ is contained in $Z_+$ either. To see this, observe that the locus $Z^{(j-1)}$ is cut out by $j-1$ elements of degree one, so by Krull's height theorem
$\dim Y \ge n-(j-1) = n-j+1$ for every component $Y$ of $Z^{(j-1)}$. Note that $j \le q$, hence $n-j+1 > n-q = p-1$, while $\dim Z_+ = p-1$ by Lemma~\ref{lem:Z}(b) and a closed subset of $Z_+$ has dimension at most $p-1$.
Hence every component of $Z^{(j-1)}$ appears among $Y_1,\dots,Y_r$.
Thirdly, the ring $S_P$ is a finitely generated $\CC$-algebra, so $X$ is Noetherian and the list $Y_1,\dots,Y_r$ is finite.

For each $i$, the component $Y_i$ does not lie in $Z_+$. Since $Z_+$ is the base locus of $W_-$, there is an element of $W_-$ that does not vanish identically on $Y_i$. Therefore, the elements of $W_-$
that do vanish on $Y_i$ form a proper subspace. There are finitely many components and the field $\CC$ is infinite, so we may choose $\theta_j$ outside the
union of these subspaces.
By Lemma~\ref{lem:toric}(c), each component of each intersection $Y_i\cap V(\theta_j)$ has dimension one less than that of $Y_i$. We check the two conditions for $j$ in turn. Every component of $Z^{(j)}$ lies in some component $Y$ of $Z^{(j-1)}$, and is then a component of the intersection $Y\cap V(\theta_j)$, so has dimension $\dim Y - 1 \le n-j$ by the first condition for $j-1$. Likewise, every component of $Z_-\cap Z^{(j)}$ lies in some component $Y'$ of $Z_-\cap Z^{(j-1)}$, and is then a component of the intersection $Y'\cap V(\theta_j)$, so has dimension $\dim Y' - 1 \le \dim Z_- - j$ by the second condition for $j-1$. So both conditions hold for $j$, which completes the induction.

Let us now consolidate the first part. Write $Z := Z^{(q)}$ for the locus after the last step. The components of the intersection $Z_-\cap Z$ have dimension at most $\dim Z_- - q < 0$, so that intersection is empty and the locus $Z$ avoids $Z_-$. Take a component $Y$ of the locus $Z$. If $Y$ lies in $Z_+$, then $\dim Y \le \dim Z_+ = p-1$ by Lemma~\ref{lem:Z}(b). Otherwise the first condition gives $\dim Y \le n-q = p-1$. So every component of $Z$ has dimension at most $p-1$. Since $Z$ avoids the base locus $Z_-$ of $W_+$, no component of $Z$ is contained in $Z_-$. So, for each component of $Z$, the elements of $W_+$ vanishing identically on that component form a proper subspace of $W_+$.

Next, we choose the elements $\phi_1,\dots,\phi_p\in W_+$ one at a time. Write $Z_k := Z\cap V(\phi_1,\dots,\phi_k)$ for the locus after $k$ steps, with $Z_0 = Z$.
Suppose that $\phi_1,\dots,\phi_k$ have been chosen for some $k \ge 0$. The locus $Z_k$ lies in $Z$ and avoids $Z_-$, hence no component of it is contained in $Z_-$, and the elements of $W_+$ vanishing identically on a given component form a proper subspace of $W_+$.
As before, there are finitely many components, so we may choose $\phi_{k+1}$ outside the union of
those subspaces. By Lemma~\ref{lem:toric}(c) every component of $Z_{k+1}$ then has dimension one less than that of a component of $Z_k$.

Every component of $Z_0$ has dimension at most $p-1$, so every component of $Z_k$ has
dimension at most $p-1-k$. Hence after $p-1$ steps every component of $Z_{p-1}$ is a single point, and no such point lies in $Z_-$, so the element $\phi_p$ chosen as above does not vanish there. By Lemma~\ref{lem:toric}(c) applied with $\dim Y = 0$, the intersection $Z_{p-1}\cap V(\phi_p)$ is empty. Hence the locus $Z_p$ is empty.

Let $V$ be the span of $\theta_1, \dots, \theta_q, \phi_1, \dots, \phi_p \in (S_P)_1$. The locus $Z_p$ is the common vanishing locus on $X$ of those $p+q$ elements, and $Z_p$ is empty, so by Lemma~\ref{lem:toric}(b) the quotient $S_P/(V)$ is finite-dimensional. By Krull's height theorem, the ideal generated by $V$ cannot be generated by fewer than $p+q = n+1$ elements, hence the subspace has dimension $\dim V = n+1$. By construction $V_- = \Span_{\CC}(\theta_1, \dots, \theta_q) \subseteq W_-$ and $V_+ = \Span_{\CC}(\phi_1, \dots, \phi_p) \subseteq W_+$, so the two summands satisfy $\dim V_- \le q$ and $\dim V_+ \le p$. Their sum is $\dim V = p+q$, hence both bounds are equalities. So the class in the representation ring is $[V] = p\chi_0 + q\chi_1 = [\rho\otimes\CC]$, and we have shown that if $P$ has no obstructing face then $S_P$ has a $\sigma$-stable l.s.o.p.\ $V$ whose class in $\MR(G)$ is $[V] = [\rho \otimes \CC]$.

\medskip
\noindent $(\Leftarrow)$
Conversely, assume that $P$ has an obstructing face $F$. Suppose for contradiction that $V$ is a $\sigma$-stable l.s.o.p.\ of class $[V] = [\rho\otimes\CC]$, so that the anti-invariant summand has dimension $\dim V_- = q$ by the splitting above. The face $F$ is $\sigma$-odd and satisfies $\dim F^\sigma \ge q$, so by Lemma~\ref{lem:Z}(c) the locus $Z_-$ has dimension at least $q$. Let $Y$ be a component of $Z_-$ with dimension at least $q$, and fix a basis $\theta_1,\dots,\theta_q$ of $V_-$. 
By applying Lemma~\ref{lem:toric}(c) repeatedly, we show that $Y \cap V(\theta_1, \dots, \theta_q)$ is nonempty. Explicitly, let $Y_0 := Y$ and, for each $1 \le k \le q$, let $Y_k$ be a component of
$Y_{k-1}\cap V(\theta_k)$. At each step $\dim Y_{k-1} \ge q-k+1 \ge 1$, so by
Lemma~\ref{lem:toric}(c), the intersection is nonempty and we have
$\dim Y_k \ge \dim Y_{k-1} - 1$. Hence $\dim Y_k \ge q-k$, so the set $Y_q$ is
nonempty and is contained in $Y\cap V(\theta_1,\dots,\theta_q)$. So take any $\xi \in Y \cap V(\theta_1, \dots, \theta_q)$. Since $\theta_1, \dots, \theta_q$ is a basis for $V_-$, we have that every element of $V_-$ vanishes at $\xi$. But $Y \subseteq Z_- = V(W_+)$ and so every element of $W_+$ vanishes at $\xi$. In particular $V_+ \subseteq W_+$, so every element of $V = V_+ \oplus V_-$ vanishes at $\xi$. Since $\xi$ is a common zero of a basis of $V$, by Lemma~\ref{lem:toric}(b), the
quotient $S_P/(V)$ is not finite-dimensional. This contradicts the hypothesis that
$V$ is an l.s.o.p. Hence, when $P$ has an obstructing face, the ring $S_P$ has no $\sigma$-stable l.s.o.p.\ whose class in $\MR(G)$ is $[\rho \otimes \CC]$.

\medskip

For the final part, if a $\sigma$-stable l.s.o.p.\ $V$ with $[V]= [\rho\otimes\CC]$ exists, by Proposition~\ref{prop:lsop-eff}, we have that $\varphi[t]$ is the graded character of $S_P/(V)$ and, in particular, that $\varphi[t]$ is effective.
\end{proof}

\begin{example}\label{ex:lsop-in-action}
Consider the diagonally reflected square in Example~\ref{ex:swap-square} and for ease of notation write the four vertex monomials as
\[
a = x^{(0,0)},\quad b = x^{(1,0)},\quad c = x^{(0,1)},\quad d = x^{(1,1)},
\qquad\text{so }
S_P = \CC[a,b,c,d]/(ad-bc),
\]
where $\sigma$ swaps variables $b$ and $c$. Here $p = 2$ and $q = 1$. In this case, $P$ has no $\sigma$-odd faces: the invariant faces are the entire polytope $P$ and the vertices $(0,0)$ and $(1,1)$. Each invariant face contains a fixed lattice point of height one, so $Z_-$ is empty and $P$ has no obstructing face. Following the proof, take one anti-invariant form and then two invariant forms:
\[
\theta = b - c, \qquad \phi_1 = a+d, \qquad \phi_2 = b+c .
\]
Each form cuts the dimension by one, as in the proof: the locus $X \cap V(\theta)$ is a curve, the locus $X\cap V(\theta,\phi_1)$ is a finite set of points, and the locus $X\cap V(\theta,\phi_1,\phi_2)$ is empty. Solving these linear equations gives us $b = c = 0$ and $a = -d$, and the relation $ad = bc$ in $S_P$ becomes $d^2 = 0$. And so the quotient $S_P/(V) \cong \CC[d]/(d^2)$ is finite-dimensional, and the span $V$ is an l.s.o.p.\ of class $[V] = 2\chi_0 + \chi_1$. The involution $\sigma$ fixes $d = x^{(1,1)}$, so the quotient decomposes as $\CC[d]/(d^2) = \CC \cdot 1 \oplus \CC \cdot d$ with $\sigma$ acting trivially on each summand. Hence Proposition~\ref{prop:lsop-eff} gives $\varphi[t] = (1+t)\chi_0$, which agrees with Example~\ref{ex:poly-in-action}.

Next consider the pyramid in Example~\ref{ex:pyramid}, which has five lattice points: the apex at $0$, together with the four points
\[
A = e_x, \qquad B = e_y, \qquad C = e_x+e_z, \qquad D = e_y+e_z
\]
that are the vertices of the $\sigma$-odd facet $F$. So the degree-one part $(S_P)_1$ has basis $x^0, x^A, x^B, x^C, x^D$ and $X$ lies in $\PP^4$. The group element $\sigma$ fixes $0$ and swaps $A$ with $B$ and $C$ with $D$, hence
\[
W_+ = \big\langle x^0,\ x^A+x^B,\ x^C+x^D \big\rangle, \qquad
W_- = \big\langle x^A-x^B,\ x^C-x^D \big\rangle ,
\]
which have dimension three and two respectively. So the class of the degree-one part of $S_P$ in $\MR(G)$ is $[(S_P)_1] = 3\chi_0 + 2\chi_1$. Note that an l.s.o.p.\ $V$ of the correct class has
$[V] = 3\chi_0 + \chi_1$. So any such $V$ decomposes as a sum of $V_+ = W_+$ and $V_- = \langle\theta\rangle$ for a single element $\theta$ of $W_-$.

The vertices of $F$ satisfy $A + D = B + C$, so the corresponding
coordinates satisfy $x^Ax^D = x^Bx^C$ on $X$. Every element of $W_+$ vanishes at the
points
\[
[\,0 : s : -s : u : -u\,] \in \PP^4, \qquad \text{for any } [s:u] \in \PP^1 ,
\]
which lie in the hyperplane $\{x^0 = 0\}$ and satisfy the relation coming from $F$. So these points form a distinguished curve inside $Z_-$. That curve is the component $Y$ of $Z_-$ that Lemma~\ref{lem:Z}(c) supplies for the obstructing facet $F$, of dimension $\dim F^\sigma = 1 \ge q$.
Suppose that $\theta = \lambda(x^A-x^B) + \mu(x^C-x^D)$, then its value at a point on the distinguished curve is
$2(\lambda s + \mu u)$, which vanishes for $[s:u] = [-\mu:\lambda]$. So any candidate for $V$, i.e., a choice of $\theta$, has a common zero on that curve. So by Lemma~\ref{lem:toric}(b), there is no l.s.o.p.\ of class $[\rho\otimes\CC]$. Since $q = 1$, there is only one anti-invariant form available for us to choose, which cannot \emph{cut away} the distinguished curve. 
\end{example}

\section{Refinements and conjectures}\label{sec:conjectures}

Throughout $G = \langle\sigma\rangle$ is a group of order two and we use the notation from Sections~\ref{setup} and \ref{sec:order-two}. 
In this section, we consider how to detect if $\varphi[t]$ is polynomial from the numerator of the Ehrhart series of $P^\sigma$ and give an inequality that always holds for the coefficients of this numerator. We give a small refinement of Theorem~\ref{thm:lsop} in terms of a single $\sigma$-odd face, and we end with two conjectures and a question.

\subsection{Numerator of the fixed polytope}
The \emph{denominator} of a rational polytope $Q$ is the least positive integer $D$ such that the dilate $DQ$ is a lattice polytope. By Lemma~\ref{lem:fixed}(a), the fixed polytope $P^\sigma$ has denominator dividing two, which gives the following.

\begin{lemma}\label{lem:Hform}
The Ehrhart series of $P^\sigma$ is given by $\Ehr_{P^\sigma}(t) = H(t)/(1-t^2)^p$ for some polynomial $H$ with nonnegative integer coefficients. It satisfies
\[
(1+t)^{\,p}\,\varphi[t](\sigma) \ =\ (1+t)^{\,q}\,H(t),
\]
and $\varphi[t]$ is a polynomial if and only if $(1+t)^r$ divides $H$, where $r = \max(p-q, 0)$.
\end{lemma}

\begin{proof}
By Lemma~\ref{lem:fixed}(a) and (c), the dilate $2P^\sigma$ is a lattice polytope of dimension $p-1$ with vertices in the fixed lattice $\Lambda$. So the denominator of $P^\sigma$ divides two, and the Ehrhart series of $P^\sigma$ has the form $H(t) / (1-t^2)^p$ for some polynomial $H$ with nonnegative integer coefficients. See \cite[Theorem~2.8]{Sta80}. By Equation~\eqref{eq:evals}, we have
\[
\varphi[t](\sigma) = (1-t)^p(1+t)^q\,\frac{H(t)}{(1-t)^p(1+t)^p} = \frac{(1+t)^{q}H(t)}{(1+t)^{p}},
\]
which gives us the identity. For the last claim, the evaluation of $\varphi[t]$ at the identity is a polynomial, so $\varphi[t]$ is polynomial if and only if the evaluation $\varphi[t](\sigma)$ is polynomial. If $p-q \le 0$, then by Corollary~\ref{cor:poly}(i) $\varphi[t]$ is polynomial and $(1+t)^r = 1$ trivially divides $H$. On the other hand, if $p-q > 0$, then $\varphi[t](\sigma) = (1+t)^{q-p}H(t)$ is polynomial if and only if $(1+t)^{p-q}$ divides $H$.
\end{proof}

For groups of order two, the Effectiveness Conjecture essentially says that if $H$ is divisible by $(1+t)^r$ then the coefficients of $h^*_P$ dominate those of $(1+t)^{q-p}H$. We have seen in Examples~\ref{ex:counterexample} and \ref{ex:counterexample-c2} that the domination can fail. In \cite{Sta23}, Stapledon defines, for each positive integer $N$, the $h^*_N$-series
\[
h^*_N[t] \ :=\ \det(I-t\rho)\EES_P^G\left(t^{1/N}\right),
\]
which is a power series in $t^{1/N}$ with coefficients in the representation ring $\MR(G)$. For $N = 1$ we recover $\varphi[t]$, and Stapledon shows that $h^*_N$ is effective whenever the order of the group divides $N$, see \cite[Remark~3.9]{Sta23}. For a group of order two we may take $N = 2$, so $h^*_2$ is effective for every $\sigma$-invariant lattice polytope. Effectiveness of $h^*_2$ gives the following inequality of coefficients, whose proof works with the variable $s = t^{1/2}$.

\begin{proposition}\label{prop:known-ineq}
Let $P$ be a $\sigma$-invariant lattice polytope. Then for each $i$, we have
\[
\left[(1+t)^{\,p+q}\,h^*_P(t)\right]_i \ge \left[(1+t^2)^{\,q}\,H(t)\right]_i
\]
where $[f(t)]_i$ denotes the coefficient of $t^i$ in the polynomial $f$.
\end{proposition}

\begin{proof}
Write $s$ for a formal variable such that $t = s^2$. So we have $h^*_2[s^2] = \det(I-s^2\rho)\,\EES_P^G(s)$, and evaluating it at the identity gives
\[
h^*_2[s^2](e) = (1-s^2)^{n+1}\frac{h^*_P(s)}{(1-s)^{n+1}} = (1+s)^{n+1}h^*_P(s).
\]
By Lemma~\ref{lem:Hform}, evaluating $h^*_2[s^2]$ at $\sigma$ gives us 
$h^*_2[s^2](\sigma) = (1+s^2)^qH(s)$. 
By \cite[Remark~3.9]{Sta23}, the series $h^*_2[s^2]$ is effective, so by Proposition~\ref{prop:eff} the two evaluations satisfy $\left[(1+s)^{n+1}h^*_P(s)\right]_i \ge \left[(1+s^2)^{q}H(s)\right]_i$ for each $i$. Recall that $p+q = n+1$, so renaming the variable $s$ as $t$ gives the inequality of the statement.
\end{proof}

\subsection{Obstructions from faces}
Recall condition (3$'$) of Question~\ref{q:433}: the polytope $P$ has a $\sigma$-invariant lattice polyhedral subdivision $\mathcal S$ whose toric face ring $S_{\mathcal S}$ has a $\sigma$-stable l.s.o.p.\ of class $[\rho\otimes\CC]$. By \cite[Corollary~4.32]{Sta23}, condition (3$'$) sits between condition (3) and condition (4) of Stapledon's hierarchy (Theorem~\ref{thm:stapledon}). In the proposition below, we give a refinement of Theorem~\ref{thm:lsop} that applies to such invariant subdivisions.

For a nonempty face $F$ of $P$, regard $F$ as a $\sigma$-invariant lattice polytope in its own affine span, with ambient lattice $L_F$, and write $S_F$ for its Ehrhart ring and $X_F = \Proj S_F$ for its projective toric variety. The polytope $F$ is full-dimensional in that span and lies at height one, so $F$ satisfies the setup of Section~\ref{sec:complete_setup} with $\ZZ^{n+1}$ replaced by $L_F$, and the results of Section~\ref{sec:proofs} apply to $F$. The numbers $p$ and $q$ are replaced by the dimensions of the $(\pm1)$-eigenspaces of $\sigma$ acting on $\Span_\RR F$, which are in general smaller than $p$ and $q$. Write $Z_-(X_F)$ for the locus in $X_F$ cut out by the invariant elements of degree one.

\begin{proposition}[A local obstruction at a single face]\label{prop:subdiv}
Let $F$ be an obstructing face of $P$.

\begin{enumerate}[label=(\alph*)]
    \item If $V' \subseteq (S_F)_1$ is a $\sigma$-stable subspace whose anti-invariant part has dimension at most $q$, then the quotient $S_F/(V')$ is infinite-dimensional.
    
    \item Assume that $F$ has no proper $\sigma$-invariant lattice subdivision. Then for any $\sigma$-invariant lattice polyhedral subdivision $\mathcal S$ of $P$, the toric face ring $S_{\mathcal S}$ does not admit a $\sigma$-stable l.s.o.p.\ $V \subseteq (S_{\mathcal S})_1$ whose class in $\MR(G)$ is $[V] = [\rho\otimes\CC]$.
\end{enumerate}
\end{proposition}

\begin{proof} Throughout this proof, the letter $\Gamma$ denotes a face of $F$ or a cell of a subdivision.

(a) Let $V' \subseteq (S_F)_1$ be a $\sigma$-stable subspace whose anti-invariant part has dimension at most $q$. We first find a component $Y$ of the invariant base locus of $X_F$ with $\dim Y = \dim F^\sigma$, and then cut $Y$ down by a basis of the anti-invariant part $V'_-$ to produce a common zero of $V'$.

Viewed in its own affine span, the face $F$ is still $\sigma$-odd, and the fixed part satisfies $\dim F^\sigma \ge q$ because $F$ is obstructing.
By Lemma~\ref{lem:Z}(c), the locus $Z_-(X_F)$ meets the dense orbit $\orb{F}$ of $X_F$ in a set of dimension $\dim F^\sigma$. A proper face $\Gamma$ of $F$ that is not $\sigma$-invariant is not $\sigma$-odd, so $Z_-(X_F)$ misses the orbit $\orb\Gamma$ by Lemma~\ref{lem:Z}(c). A proper $\sigma$-invariant face $\Gamma$ of $F$ satisfies $\dim\Gamma^\sigma < \dim F^\sigma$ by Lemma~\ref{lem:fixed}(c), so $Z_-(X_F)$ meets the orbit $\orb\Gamma$ in a set of strictly smaller dimension. So every orbit other than $\orb F$ contributes a set of dimension less than $\dim F^\sigma$. Hence the locus $Z_-(X_F)$ is nonempty of dimension $\dim F^\sigma$, and we select some irreducible component $Y$ of $Z_-(X_F)$ of that dimension.

The subspace $V'$ splits as $V' = V'_+\oplus V'_-$, and the anti-invariant part $V'_-$ has dimension $r \le q$ by hypothesis. The invariant part $V'_+$ vanishes on $Y$ because $Y \subseteq Z_-(X_F)$.
Let $\theta_1, \dots, \theta_r$ be a basis of $V'_-$. Let $Y_0 = Y$ and for each $k$ from $1$ to $r$ let $Y_k$ be a component of $Y_{k-1}\cap V(\theta_k)$. By induction, we check that such $Y_k$ exists and satisfies $\dim Y_k \ge \dim F^\sigma - k$. For $k = 0$, the component $Y$ satisfies the dimension bound. If $k \ge 1$, then by induction the bound holds for $k-1$ and so  
\[
\dim Y_{k-1} \ \ge\ \dim F^\sigma - k + 1 \ \ge\ \dim F^\sigma - q + 1 \ \ge\ 1.
\]
By Lemma~\ref{lem:toric}(c), the intersection $Y_{k-1}\cap V(\theta_k)$ is nonempty and each of its components has dimension at least $\dim Y_{k-1} - 1 \ge \dim F^\sigma - k$.
In particular, we have that $Y_r$ is nonempty. Observe that every element of $V'$ vanishes on $Y_r$: each element of the basis $\{\theta_i\}$ vanishes on $Y_r$, hence $V'_-$ vanishes on $Y_r$, and $V'_+$ vanishes on $Y_r$ because it vanishes on $Y$. Hence $V'$ has a common zero on $X_F$, and by Lemma~\ref{lem:toric}(b) we have that $S_F/(V')$ is infinite-dimensional.

\smallskip

(b) Let $\mathcal S$ be a $\sigma$-invariant lattice polyhedral subdivision of $P$, and assume that the face $F$ has no proper $\sigma$-invariant lattice subdivision. We first show that $F$ is a cell of $\mathcal S$. We then construct a $\sigma$-equivariant surjection of graded rings $\pi_F \colon S_{\mathcal S}\to S_F$, so that an l.s.o.p.\ of class $[\rho\otimes\CC]$ for $S_{\mathcal S}$ would give a subspace of $(S_F)_1$ of the kind excluded by part (a).

Since $\mathcal S$ is a $\sigma$-invariant subdivision of $P$, its restriction to $F$ is also a $\sigma$-invariant subdivision of $F$. By assumption $F$ has no proper $\sigma$-invariant lattice subdivision, so $F$ is a cell of that restriction and hence a cell of $\mathcal S$. 
Let $\pi_F \colon S_{\mathcal S} \rightarrow S_F$ be the map of graded vector spaces that sends a basis monomial $x^u$ to itself if the point $u$ lies in the cone $C_F$, and to zero otherwise.

The map $\pi_F$ is a surjective homomorphism of graded rings. For surjectivity, the face $F$ is a cell of $\mathcal S$, so each basis monomial $x^u$ of $S_F$ is also a basis monomial of $S_{\mathcal S}$ and hence lies in the image.
Note that both rings are graded by height, and $\pi_F$ preserves the grading because it sends each basis monomial either to itself or zero. We argue that $\pi_F$ is multiplicative as follows. We first check that the intersection $C_F\cap C_\Gamma$ is a face of the cone $C_\Gamma$ for every cell $\Gamma$ of $\mathcal S$.

Write $F = P \cap \{\ell = \max_P\ell\}$ for a supporting functional $\ell\colon\RR^{n+1}\to\RR$, and let $\Gamma$ be a cell of $\mathcal S$. If the cell $\Gamma$ does not meet $F$, then $C_F\cap C_\Gamma = \{0\}$, because a common point of positive height would give, after scaling to height one, a common point of $F$ and $\Gamma$. So the intersection $C_F \cap C_\Gamma$ is a face of $C_\Gamma$ in that case. Now suppose that $\Gamma$ meets $F$. Some point of $\Gamma$ attains the maximum value $\max_P\ell$, so $\max_\Gamma\ell = \max_P\ell$ and
\[
F\cap \Gamma \ =\ \Gamma \cap \{\ell = \max_\Gamma\ell\} ,
\]
which is a face of $\Gamma$. By passing to the cones, the intersection $C_F\cap C_\Gamma$ is the cone over that face, hence a face of $C_\Gamma$.

Now take two basis monomials $x^u$ and $x^v$ of $S_{\mathcal S}$ and take cases according to whether both of the points $u$ and $v$ lie in the cone $C_F$. If both $u$ and $v$
lie in the cone $C_F$, then they lie over the common cell $F$, so
$x^ux^v = x^{u+v}$ in $S_{\mathcal S}$, and the sum $u+v$ lies in $C_F$ as well and so $\pi_F(x^u) \pi_F(x^v) = x^{u+v} = \pi_F(x^{u+v})$.

Suppose now that one of the points, say $u$, lies outside $C_F$. Then $\pi_F(x^u)\pi_F(x^v) = 0$, so it remains to show that $\pi_F(x^ux^v) = 0$ as well, and here we take cases on whether the points $u$ and $v$ lie over a common cell. If the two points lie
over no common cell, then $x^ux^v = 0$. Finally, if they lie over a common cell
$\Gamma$ of $\mathcal S$, then $x^ux^v = x^{u+v}$, so it is enough to show that the sum $u+v$ lies outside $C_F$. Suppose instead that $u+v$ lies in $C_F$. Then $u+v$ lies in $C_F\cap C_\Gamma$, which is a face of $C_\Gamma$. Since $u$ and $v$ lie in $C_\Gamma$ and their sum lies in a face of $C_\Gamma$, both $u$ and $v$ lie in that face, hence $u \in C_F$, a contradiction. So the sum $u+v$ lies outside $C_F$ and $\pi_F(x^ux^v) = 0$ in this case.
Finally $\pi_F$ is equivariant because $\sigma(F) = F$ so $u \in C_F$ if and only if $\sigma u \in C_F$.

Suppose that the toric face ring $S_{\mathcal S}$ had a $\sigma$-stable l.s.o.p.\ $V$ of class $[V]= [\rho \otimes \CC]$, and let $V' = \pi_F(V)$ be its image. The map $\pi_F$ is equivariant, so $V'$ is a $\sigma$-stable subspace of $(S_F)_1$ whose anti-invariant part is the image of $V_-$ and so has dimension at most $\dim V_- = q$. Since $V$ is an l.s.o.p.\ of $S_{\mathcal S}$, the quotient $S_{\mathcal S}/(V)$ is finite-dimensional. Surjectivity of $\pi_F$ gives a surjection $S_{\mathcal S}/(V) \to S_F/(V')$, so the quotient $S_F/(V')$ is finite-dimensional as well, which contradicts part (a). Hence the toric face ring $S_{\mathcal S}$ has no $\sigma$-stable l.s.o.p.\ of class $[\rho\otimes\CC]$.
\end{proof}

The following corollary settles one implication of Question~\ref{q:433}.

\begin{corollary}[Effective with no invariant l.s.o.p.\ after subdivision]\label{cor:q433}
For the polytope $P = T\times[0,1]\times\{1\}$ in Example~\ref{ex:TxI}, the series $\varphi[t]$ is effective, the Ehrhart ring $S_P$ has no $\sigma$-stable l.s.o.p.\ of class $[\rho\otimes\CC]$, and no $\sigma$-invariant lattice polyhedral subdivision of $P$ has a toric face ring with a $\sigma$-stable l.s.o.p.\ of class $[\rho\otimes\CC]$. So condition (4) of Theorem~\ref{thm:stapledon} does not imply condition (3$'$) of Question~\ref{q:433}.
\end{corollary}

\begin{proof}
From Example~\ref{ex:TxI}, we have that $\varphi[t]$ is effective and by Theorem~\ref{thm:lsop} we have that its Ehrhart ring does not have a $\sigma$-stable l.s.o.p.\ of class $[\rho\otimes\CC]$.
Let $F$ be the obstructing two-dimensional face of $P$ from Example~\ref{ex:TxI}. The face $F$ is a unit square, so its only proper lattice subdivisions are into two triangles joined along a diagonal. However $\sigma$ interchanges the two diagonals of $F$, so no proper lattice subdivision of $F$ is $\sigma$-invariant. 
So, by Proposition~\ref{prop:subdiv}(b), no $\sigma$-invariant lattice polyhedral subdivision of $P$ has a toric face ring with a $\sigma$-stable l.s.o.p.\ of class $[\rho\otimes\CC]$. Hence the polytope $P$ satisfies condition (4) of Theorem~\ref{thm:stapledon} but not condition (3$'$) of Question~\ref{q:433}.
\end{proof}

\subsection{Conjectures}
The first conjecture is a strengthening of the inequality in Proposition~\ref{prop:known-ineq}. Recall from Lemma~\ref{lem:Hform} that for a $\sigma$-invariant lattice polytope $P$ the Ehrhart series of $P^\sigma$ is the rational function $H(t)/(1-t^2)^p$ for a polynomial $H$ with nonnegative integer coefficients.

\begin{conjecture}[Coefficient-wise inequality]\label{conj:ineq}
For any $\sigma$-invariant lattice polytope $P$, we have
\[
\left[(1+t)^{p-q}h^*_P(t)\right]_i \ge [H(t)]_i \quad \text{if } p \ge q
\qquad
\text{and}
\qquad \left[h^*_P\right]_i \ge [(1+t)^{q-p}H(t)]_i \quad \text{if } p < q,
\]
where $[f(t)]_i$ denotes the coefficient of $t^i$ of the polynomial $f(t)$.
\end{conjecture}

Note that multiplying the polynomials on each side of the inequality by $(1+t)$ preserves the inequality, but dividing does not. For instance, take $A = 1+t^2$ and $B = t$ and observe that $(1+t)(A-B) = 1+t^3$ has nonnegative coefficients, so $(1+t)A$ and $(1+t)B$ are comparable but $A$ and $B$ are not.
In particular, for Example~\ref{ex:counterexample} where $p = 4 > 3 = q$, we have
\[
(1+t)h^*_P = 1 + 2t + t^2 + t^3 + 2t^4 + t^5
\quad \text{and} \quad
H = 1 + t^2 + t^3 + t^5,
\]
so the inequality in Conjecture~\ref{conj:ineq} holds. Below we explain that the conjecture is a strengthening of the previous inequality and explain the natural cases where it holds.

\begin{proposition}[Known cases of the inequality]\label{prop:ineq-known}
Conjecture~\ref{conj:ineq} implies Proposition~\ref{prop:known-ineq}. The conjecture holds if $\varphi[t]$ is a polynomial and effective. The conjecture also holds for Example~\ref{ex:counterexample}, where $\varphi[t]$ is a polynomial but is not effective.
\end{proposition}

\begin{proof}
We first show that Conjecture~\ref{conj:ineq} implies Proposition~\ref{prop:known-ineq}, so suppose that the conjecture holds for $P$. Multiplying the polynomials on each side of the conjectured inequality by the same power of $(1+t)$ preserves the inequality. The conjecture has two cases, according to whether $p \ge q$, and in both cases multiplying each side by $(1+t)^{\min(p,q)+q}$ gives $[(1+t)^{p+q}h^*_P]_i \ge [(1+t)^{2q}H]_i$ for each $i$. So, to show that the conjecture implies Proposition~\ref{prop:known-ineq}, it suffices to show that $\left[(1+t)^{2q}H\right]_i \ge \left[(1+t^2)^{q}H\right]_i$ for each $i$. Suppose that two pairs of polynomials satisfy $A_1 \ge B_1 \ge 0$ and $A_2 \ge B_2 \ge 0$ coefficient-wise. Then
\[
A_1A_2 - B_1B_2 \ =\ (A_1-B_1)A_2 + B_1(A_2-B_2)
\]
has nonnegative coefficients. Since $(1+t)^2 - (1+t^2) = 2t$, applying the identity $q-1$ times gives $(1+t)^{2q} \ge (1+t^2)^q$ coefficient-wise, and applying it once more with the polynomial $H$ of Lemma~\ref{lem:Hform}, which has nonnegative coefficients, gives the required inequality.

Suppose next that $\varphi[t]$ is a polynomial and effective. Then $[h^*_P]_i \ge [\varphi[t](\sigma)]_i$ for each $i$ by Proposition~\ref{prop:eff}. Multiplying both sides by $(1+t)^{\max(p-q,0)}$ and using the identity of Lemma~\ref{lem:Hform} gives the inequality of Conjecture~\ref{conj:ineq}, in the case $p \ge q$ and in the case $p < q$.

For the last claim, the computation before this proposition checks the inequality for Example~\ref{ex:counterexample}.
\end{proof}

Our other conjecture is about the effectiveness of $\varphi[t]$ for dilates of $P$.

\begin{conjecture}[Effectiveness of dilates]\label{conj:dilate}
If $\varphi_P$ is effective, then $\varphi_{mP}$ is effective for each positive integer $m$.
\end{conjecture}

If $P$ has no obstructing face, then the conjecture holds by the last claim of Proposition~\ref{prop:dilate}. By \cite[Corollary~7.14]{Sta11}, we know that $\varphi_{mP}[t]$ is effective for every even $m$ and so let us consider the odd dilates. In Example~\ref{ex:square}, we have that $\varphi_P[t]$ is not polynomial hence not effective, and by Proposition~\ref{prop:dilate} we have that $\varphi_{mP}[t]$ is not polynomial hence not effective for every odd $m$. So this suggests a dichotomy, either $\varphi_{mP}[t]$ is effective for every $m$ or $\varphi_{mP}[t]$ is effective only for even $m$. We pose this as the following question.

\begin{question}[Sets of effective dilates]\label{q:dilate-set}
For which sets $E$ of odd positive integers is there a $\sigma$-invariant lattice polytope $P$ such that $\varphi_{mP}[t]$ is effective exactly for the even positive integers $m$ together with the integers $m \in E$? The dichotomy above suggests that $E$ is either empty or the set of all odd positive integers.
\end{question}

\begin{remark}[Permutation representations and the equivariant volume]\label{rem:dead}
In many families of examples, such as simplices, it is known that the coefficients of $\varphi[t]$ are permutation representations. In our setting of a group of order two, the $i$th coefficient is a permutation representation if and only if $0 \le \varphi_i(\sigma) \le (h^*_P)_i$. We have seen examples that do not satisfy this condition. For instance, the polytope of Example~\ref{ex:circuit} has $\varphi[t] = 1 + \chi_1 t + t^2$, whose linear coefficient fails the condition $0 \le \varphi_1(\sigma)$. And in Example~\ref{ex:counterexample}, the condition $\varphi_2(\sigma) \le (h^*_P)_2$ fails.

Let us compare these observations with \cite[Conjecture~12.2]{Sta11}, which says that if $\varphi[t]$ is effective then its evaluation at $1$, called the \emph{equivariant volume}, is a permutation representation. Note that the conjecture holds immediately for any cyclic group of prime order by an observation of Stapledon. However, in both Example~\ref{ex:counterexample} and Example~\ref{ex:counterexample-c2} the equivariant volume is $\varphi[1] = 3\chi_0 + \chi_1$, which is a permutation representation even though $\varphi[t]$ is not effective. That coincidence might lead us to ask whether polynomiality of $\varphi[t]$ implies that $\varphi[1]$ is a permutation representation. We note that a positive answer would be a consequence of Conjecture~\ref{conj:ineq}, as follows.
Evaluating the identity of Lemma~\ref{lem:Hform} at $t = 1$ gives $H(1) = 2^{\,p-q}\varphi[1](\sigma)$, and the polynomial $H$ has nonnegative coefficients, hence $\varphi[1](\sigma) \ge 0$. Evaluating Conjecture~\ref{conj:ineq} at $t = 1$ then gives $\varphi[1](\sigma)\le h^*_P(1) = \varphi[1](e)$. Write $\varphi[1] = \lambda\chi_0 + \mu\chi_1$, so that $\lambda+\mu = \varphi[1](e)$ and $\lambda-\mu = \varphi[1](\sigma)$. The two inequalities give $\lambda \ge \mu \ge 0$, and polynomiality of $\varphi[t]$ makes both $\lambda$ and $\mu$ integers, hence $\varphi[1]$ is a permutation representation. Integrality is needed here: for the orthogonally reflected square of Example~\ref{ex:square}, the rational function $\varphi[t](\sigma)$ takes the value $1$ at $t = 1$ while $h^*_P(1) = 2$, so the two inequalities hold for $\lambda = \frac32$ and $\mu = \frac12$, which are not the multiplicities of a representation.
\end{remark}

\section{Other counterexamples}

In the final stages of preparing this paper, the author was made aware of two counterexamples to Conjecture~\ref{conj:eff} that had recently appeared in preprints.
For each example we use the notation of Section~\ref{setup}, with the polytope placed at height one.
 
\begin{example}[A counterexample of Chalmers {$S_3$} \cite{Cha26}]\label{ex:chalmers}
Let $n = 4$ and write the coordinates of $\ZZ^5$ as $(x_1,x_2,x_3,w,h)$ with $\widetilde M_0 = \{h=0\}$ and fixed vector $\alpha = (0,0,0,0,1)$. Let $G = S_3$ act by permuting the coordinates $x_1,x_2,x_3$ and fixing $w$. Write $e_1,e_2,e_3$ for the standard basis of $\ZZ^3$ and fix $u = (1,1,1)$. Define the polytope $P$ with seven vertices:
\[
P := \conv\left(\{(0,0,0,1,1)\} \cup \{(-e_i,\,-1,\,1),\, (u-e_i,\,-1,\,1) \colon \text{for } 1 \le i \le 3\}\right) \subseteq \widetilde M.
\]
Then we have:
\begin{itemize}[]
    \item the polytope $P$ is $4$-dimensional and reflexive, with seven vertices and
    normalised volume $\Vol(P) = 18$,
    
    \item the $h^*$-polynomial is $h^*_P = 1 + 4t + 8t^2 + 4t^3 + t^4$,
    \item let $\chi_0$, $\chi_1$, and $\chi_2$ be the trivial, sign, and standard representations of $S_3$, then the equivariant $H^*$-series is
    \[
    \varphi[t] = \chi_0 + (2\chi_0 + \chi_2) t + (5\chi_0 - \chi_1 + 2\chi_2) t^2 + (2\chi_0 + \chi_2) t^3 + \chi_0 t^4 ,
    \]
    \item so $\varphi[t]$ is a polynomial but its the coefficient of $t^2$ is not effective.
\end{itemize}
We note that restricting the action to any subgroup generated by a transposition $\sigma$ (the choice is not important as all transpositions are conjugate) gives $\varphi[t](\sigma) = 1 + 2t + 6t^2 + 2t^3 + t^4$, hence
\[
\varphi[t] = (1 + 3t + 7t^2 + 3t^3 + t^4)\chi_0 + (t + t^2 + t^3)\chi_1
\]
for $\langle \sigma \rangle$, which is effective. 
\end{example}
 
\begin{example}[Counterexamples Ziegler \cite{Zie26}]\label{ex:ziegler}
Let $Q \subseteq \RR^r$ be an $r$-dimensional lattice polytope and write the coordinates of $\ZZ^{r+4}$ as $(x,y,z,u,h)$ with $z \in \RR^r$. Let $\widetilde M_0 = \{h = 0\} \subseteq \RR^{r+4}$ and $\alpha = (0,0,0,0,1)$. Let $R(Q) \subseteq \widetilde M$ be the polytope:
\[
R(Q) = \conv\left(
\{(a,b,0,1,1) \colon a,b \in \{0,1\}\}
\cup 
\{(0,0,v,0,1) \colon v \in Q\} 
\right),
\]
which is an $(r+3)$-dimensional lattice polytope. Let $\sigma$ act by
\[
\sigma(x,y,z,u,h) = (u-x,\,y,\,z,\,u,\,h).
\]
In particular $\sigma$ is an involution with type $(r+2,0,1)$, so $p = r+3$ and $q = 1$. The copy of matrix $U$ in the normal form comes from the coordinates $x$ and $u$. 
Up to lattice equivalence, the polytope $R(Q)$ is the \emph{free join} of the unit square in Example~\ref{ex:square} with $Q$, where $\sigma$ acts trivially on $Q$. Since the equivariant $H^*$-series is multiplicative under free joins \cite[Lemma~2.22]{Sta23}, the series of $R(Q)$ is the product of $h^*_Q$ with the equivariant $H^*$-series of the square.

Let $S$ be the orthogonally reflected square in Example~\ref{ex:square}. By Example~\ref{ex:poly-in-action}, we have that $\varphi_S[t](\sigma)$ has a pole of order one at $t = -1$. So if $h^*_Q(-1) = 0$, then the poles cancel. In that case, write $h^*_Q(t) = (1+t)\,\mathbf q(t)$ for some polynomial $\mathbf q(t)$, and we have
\[
\varphi_{R(Q)}[t] = (1+t+t^2)\,\mathbf q(t)\,\chi_0 \ +\ t\,\mathbf q(t)\,\chi_1 ,
\]
which is a polynomial, but is effective if and only if $\mathbf q$ has nonnegative coefficients. Ziegler gives two choices of $Q$ (one-row Hermite normal form simplices) that lead to counterexamples.
\begin{itemize}[]
    \item Let $Q = \conv\{0,\,e_1,\,e_2,\,e_3,\,e_4,\,(1,1,1,1,2)\} \subseteq \RR^5$ be a $5$-simplex with $h^*_Q = 1+t^3$ and so $\mathbf q(t) = 1-t+t^2$. Then $R(Q)$ is an $8$-dimensional lattice polytope with ten vertices, and
    \[
    h^*_{R(Q)} = 1 + t + t^3 + t^4, \quad \text{ hence } \quad
    \varphi[t] = (1+t^2+t^4)\chi_0 + (t - t^2 + t^3)\chi_1 .
    \]
    The coefficient $\chi_0 - \chi_1$ of $t^2$ is not effective. In this case, we note that lattice points of $R(Q)$ generate a subgroup of index two, so (Gen) does not hold.
    
    \item Let $Q = \conv\{0,\,e_1,\dots,e_9,\,(13,13,13,13,13,13,13,13,10,14)\} \subseteq \RR^{10}$ be a $10$-simplex with normalised volume $14$ with
    \[
    h^*_Q = 1+t+t^2+2t^3+t^4+2t^5+t^6+t^7+2t^8+t^9+t^{10},
    \]
    and so $\mathbf q(t) = 1+t^2+t^3+2t^5-t^6+2t^7+t^9$. Then $R(Q)$ is a $13$-dimensional lattice polytope with fifteen vertices and
    $
    \varphi[t] = \sum_{i=0}^{11} (c_i \chi_0 + d_i \chi_1) t^i
    $
    is a degree eleven polynomial where 
    \[
    (c_i) = (1,1,2,2,\ 2,3,1,3,\ 1,3,1,1)
    \quad \text{and} \quad
    (d_i) = (0,1,0,1,\ 1,0,2,-1,\ 2,0,1,0).
    \]
    The coefficient $3\chi_0 - \chi_1$ of $t^7$ is not effective. In this case the lattice points of $R(Q)$ affinely generate $\ZZ^{13}$, so (Gen) holds.
\end{itemize}
\end{example}
 
\medskip
 
\noindent \textbf{Relation to our results.}
Both polytopes of Example~\ref{ex:ziegler} have $q = 1$ and $p = n$, so each one is an example of case (C2) of Proposition~\ref{prop:open}. 
The square face is $\sigma$-odd and its fixed part has dimension $1 \ge q$, hence it is obstructing.
Recall that Corollary~\ref{cor:p-eq-q-plus-1} says that if $p = q+1$ and $P$ satisfies (Gen) then $\varphi[t]$ is polynomial if and only if it is effective. The $13$-dimensional polytope satisfies (Gen), which shows that the condition $p=q+1$ is necessary.
We note also that Ziegler's $8$-dimensional counterexample has the same equivariant $H^*$-series, and the same $h^*$-polynomial, as our $6$-dimensional Example~\ref{ex:counterexample}, but with two different types of action. 
For a group of order two, the smallest known dimension for a counterexample to the Effectiveness Conjecture is six. See Question~\ref{q:small}. However, for groups in general Example~\ref{ex:chalmers} shows that the minimum dimension is either three or four.

\section*{Acknowledgements}

The author would like to thank Akihiro Higashitani and Max K\"olbl for helpful communications and discussions.

\bibliographystyle{plain}
\bibliography{ref}

@article{ASVM,
  author  = {Ardila, Federico and Supina, Mariel and Vindas-Mel\'endez, Andr\'es R.},
  title   = {The equivariant {E}hrhart theory of the permutahedron},
  journal = {Proc. Amer. Math. Soc.},
  volume  = {148},
  number  = {12},
  year    = {2020},
  pages   = {5091--5107}
}

@article{BHW,
  author  = {Bey, Christian and Henk, Martin and Wills, J\"org M.},
  title   = {Notes on the roots of {E}hrhart polynomials},
  journal = {Discrete Comput. Geom.},
  volume  = {38},
  number  = {1},
  year    = {2007},
  pages   = {81--98}
}

@article{BM85,
  author  = {Betke, Ulrich and McMullen, Peter},
  title   = {Lattice points in lattice polytopes},
  journal = {Monatsh. Math.},
  volume  = {99},
  number  = {4},
  year    = {1985},
  pages   = {253--265}
}

@book{BH,
  author    = {Bruns, Winfried and Herzog, J\"urgen},
  title     = {Cohen--Macaulay Rings},
  series    = {Cambridge Studies in Advanced Mathematics},
  volume    = {39},
  publisher = {Cambridge University Press},
  year      = {1993}
}

@article{CHK,
  author  = {Clarke, Oliver and Higashitani, Akihiro and K\"olbl, Max},
  title   = {The equivariant {E}hrhart theory of polytopes with order-two symmetries},
  journal = {Proc. Amer. Math. Soc.},
  volume  = {151},
  number  = {9},
  year    = {2023},
  pages   = {4027--4041}
}

@article{Hyp,
  author  = {Clarke, Oliver and K\"olbl, Max},
  title   = {Equivariant {E}hrhart theory of hypersimplices},
  journal = {Forum Math. Sigma},
  volume  = {13},
  year    = {2025},
  note    = {Paper No. e178}
}

@book{CLS,
  author    = {Cox, David A. and Little, John B. and Schenck, Henry K.},
  title     = {Toric Varieties},
  series    = {Graduate Studies in Mathematics},
  volume    = {124},
  publisher = {American Mathematical Society},
  year      = {2011}
}

@misc{DD,
  author = {D'Al\`i, Alessio and Delucchi, Emanuele},
  title  = {Equivariant {H}ilbert and {E}hrhart series under translative group actions},
  note   = {arXiv:2312.14088},
  year   = {2025}
}

@misc{DH,
  author = {D'Al\`i, Alessio and Higashitani, Akihiro},
  title  = {Order polytopes of graded posets are gamma-effective},
  note   = {arXiv:2505.07623},
  year   = {2025}
}

@article{EKS,
  author  = {Elia, Sophia and Kim, Donghyun and Supina, Mariel},
  title   = {Techniques in equivariant {E}hrhart theory},
  journal = {Ann. Comb.},
  volume  = {28},
  number  = {3},
  year    = {2024},
  pages   = {819--870}
}

@article{Hoc72,
  author  = {Hochster, Melvin},
  title   = {Rings of invariants of tori, {C}ohen--{M}acaulay rings generated by monomials, and polytopes},
  journal = {Ann. of Math.},
  volume  = {96},
  year    = {1972},
  pages   = {318--337}
}

@article{McM78,
  author  = {McMullen, Peter},
  title   = {Lattice invariant valuations on rational polytopes},
  journal = {Arch. Math.},
  volume  = {31},
  year    = {1978},
  pages   = {509--516}
}

@article{MW05,
  author  = {McAllister, Tyrrell B. and Woods, Kevin M.},
  title   = {The minimum period of the {E}hrhart quasi-polynomial of a rational polytope},
  journal = {J. Combin. Theory Ser. A},
  volume  = {109},
  number  = {2},
  year    = {2005},
  pages   = {345--352}
}

@article{Rei57,
  author  = {Reiner, Irving},
  title   = {Integral representations of cyclic groups of prime order},
  journal = {Proc. Amer. Math. Soc.},
  volume  = {8},
  year    = {1957},
  pages   = {142--146}
}

@article{Sta78,
  author  = {Stanley, Richard P.},
  title   = {Hilbert functions of graded algebras},
  journal = {Adv. Math.},
  volume  = {28},
  year    = {1978},
  pages   = {57--83}
}

@article{Sta87,
  author  = {Stanley, Richard P.},
  title   = {On the number of faces of centrally-symmetric simplicial polytopes},
  journal = {Graphs Combin.},
  volume  = {3},
  year    = {1987},
  pages   = {55--66}
}

@article{Sta93,
  author  = {Stanley, Richard P.},
  title   = {A monotonicity property of $h$-vectors and $h^*$-vectors},
  journal = {European J. Combin.},
  volume  = {14},
  number  = {3},
  year    = {1993},
  pages   = {251--258}
}

@book{StaEC1,
  author    = {Stanley, Richard P.},
  title     = {Enumerative Combinatorics, Volume 1},
  edition   = {2nd},
  publisher = {Cambridge University Press},
  year      = {2012}
}

@article{Sta11,
  author  = {Stapledon, Alan},
  title   = {Equivariant {E}hrhart theory},
  journal = {Adv. Math.},
  volume  = {226},
  number  = {4},
  year    = {2011},
  pages   = {3622--3654}
}

@misc{Sta23,
  author = {Stapledon, Alan},
  title  = {Equivariant {E}hrhart theory, commutative algebra and invariant triangulations of polytopes},
  note   = {arXiv:2311.17273v3},
  year   = {2025}
}

@article{Ste94,
  author  = {Stembridge, John R.},
  title   = {Some permutation representations of {W}eyl groups associated with the cohomology of toric varieties},
  journal = {Adv. Math.},
  volume  = {106},
  number  = {2},
  year    = {1994},
  pages   = {244--301}
}

@article{Sta80,
  author  = {Stanley, Richard P.},
  title   = {Decompositions of rational convex polytopes},
  journal = {Ann. Discrete Math.},
  volume  = {6},
  year    = {1980},
  pages   = {333--342}
}

@Misc{Normaliz,
  author = {W. Bruns and B. Ichim and C. S{\"{o}}ger and U. von der Ohe},
  title = {Normaliz. Algorithms for rational cones and affine monoids},
  howpublished = {Available at \texttt{https://www.normaliz.uni-osnabrueck.de}}
}

@article{IR07,
  author  = {Ichim, Bogdan and R\"omer, Tim},
  title   = {On toric face rings},
  journal = {J. Pure Appl. Algebra},
  volume  = {210},
  number  = {1},
  year    = {2007},
  pages   = {249--266}
}

@misc{Cha26,
  author = {Chalmers, Logan R.},
  title  = {Stapledon's effectiveness conjecture is false},
  year   = {2026},
  note   = {Preprint, \url{https://doi.org/10.13140/RG.2.2.10790.48962}}
}

@misc{Zie26,
  author = {Ziegler, Aran S.},
  title  = {A counterexample to {S}tapledon's effectiveness conjecture},
  year   = {2026},
  note   = {Preprint, \url{https://doi.org/10.5281/zenodo.23202548}}
}

\bigskip
\footnotesize

\noindent
{\sc Oliver Clarke, Department of Mathematical Sciences, Durham University}\\
\textit{Email address}: \texttt{oliver.clarke@durham.ac.uk}\\
\textit{URL}: \texttt{https://www.oliverclarkemath.com/}\\

\end{document}